\documentclass[12pt]{amsart}

\usepackage{amsmath}
\usepackage{amsthm}
\usepackage{graphicx}
\usepackage{amssymb}
\usepackage{enumerate}
\usepackage{tikz}

\usepackage{hyperref}
\usepackage{cleveref}

\usetikzlibrary{arrows}
\usetikzlibrary{arrows.meta}
\usetikzlibrary{decorations.markings}
\tikzset{>={Latex[width=1.2mm,length=1.7mm]}}
\usepackage{tabularx}
\usepackage{mathrsfs}

\newtheorem{thm}{Theorem}[section]
\newtheorem{prop}[thm]{Proposition}
\newtheorem{cor}[thm]{Corollary}
\newtheorem{lem}[thm]{Lemma}
\newtheorem{obs}[thm]{Observation}

\newtheorem{prob}[thm]{Problem}
\theoremstyle{definition}

\numberwithin{equation}{section}

\newcommand{\sn}{\mathfrak{S}_n}

\newcommand{\mfs}[1]{\mathfrak{S}_{#1}}

\newcommand{\bn}{\mathfrak{B}_n}

\newcommand{\mfb}[1]{\mathfrak{B}_{#1}}

\newcommand{\zqq}{\mathbb{Z}[\qp12, \qm12]}

\newcommand{\qp}[2]{q^{\frac{#1}{#2}}}
\newcommand{\qm}[2]{q^{\negthinspace\Bar\,\frac{#1}{#2}}}

\newcommand{\ol}[1]{\overline{#1}}

\newcommand{\hanq}{H_n^{\msfA}(q)}
\newcommand{\hbnq}{H_n^{\msfBC}(q)}
\newcommand{\hbjq}{H_J^{\msfBC}(q)}

\newcommand{\icoh}{\mathrm{IH}}

\newcommand{\Wejm}{W^J_-}

\newcommand{\wtc}[2]{\widetilde{C}_{#1}(#2)}

\newcommand{\sumsb}[1]{\sum_{\substack{#1}}}  
\newcommand{\inv}{\textsc{inv}}

\newcommand{\defeq}{:=} 
\newcommand{\dfct}{\mathrm{dfct}}
\newcommand{\dfctbc}{\mathrm{dfct}^{\msfBC}}

\newcommand{\trunc}{\mathrm{trunc}}

\newcommand{\spn}{\mathrm{span}}

\newcommand{\src}{\mathrm{source}}
\newcommand{\snk}{\mathrm{sink}}
\newcommand{\type}{\mathrm{type}}

\newcommand{\pavoiding}{$3412$-avoiding, $4231$-avoiding }
\newcommand{\avoidsp}{avoids the patterns $3412$ and $4231${}}
\newcommand{\avoidp}{avoid the patterns $3412$ and $4231${}}
\newcommand{\avoidingp}{avoiding the patterns $3412$ and $4231${}}

\newcommand{\ntnsp}{\negthinspace}
\newcommand{\ntksp}{\negthickspace}
\newcommand{\nTksp}{\negthickspace\negthickspace}

\newcommand{\bp}{\begin{prob}}
\newcommand{\ep}{\end{prob}}

\newcommand{\PiBC}{\Pi^{\msfBC}}

\newcommand{\ssm}{\smallsetminus}

\newcommand{\net}[2]{\mathcal F^{\mathsf{#1}}(#2)}
\newcommand{\bnet}[2]{\mathcal F^{\mathsf{#1}}_\bullet(#2)}
\newcommand{\bcnet}[2]{\mathcal F^{\mathsf{#1}}_{\bullet,\circ}(#2)}

\newcommand{\msfA}{\mathsf{A}}

\newcommand{\msfBB}{\mathsf{B}}

\newcommand{\msfC}{\mathsf{C}}

\newcommand{\msfBC}{\mathsf{BC}}

\newcommand{\msfD}{\mathsf{D}}

\newcommand{\schub}[1]{\Omega_{#1}}

\begin{document}

\author[Hobbs, Parisi, Skandera, Wang]{Gavin Hobbs$^\dag$, Tommy Parisi$^\dag$,  Mark Skandera$^\dag$, \and Jiayuan Wang$^\dag$
}

\title[Generalization of Deodhar's defect statistic]{A generalization of Deodhar's defect statistic for Iwahori--Hecke algebras of type $\msfBC$}

\bibliographystyle{dart}

\date{\today}

\begin{abstract}
Let $H$ be the Iwahori--Hecke algebra corresponding to any Coxeter group.
  Deodhar's {\em defect} statistic [{\em Geom.\;Dedicata} {\bf 36}, (1990) pp.~95--119] allows one to expand products of 
  simple Kazhdan--Lusztig basis elements of $H$ in the natural basis of $H$.
  Clearwater and the third author gave a type-$\msfA$ extension
  [{\em Ann.\;Comb.}\;{\bf 25}, no.\,3 (2021) pp.~757--787]
  of this formula 
  which combinatorially 
  describes the natural expansion of products of 
  Kazhdan--Lusztig basis elements 
  indexed by ``smooth"
  elements of the symmetric group.
  We similarly give a type-$\msfBC$ extension of Deodhar's result which combinatorially describes the natural expansion of Kazhdan--Lusztig basis elements indexed by hyperoctahedral group elements which are ``simultaneously smooth" in types $\msfBB$ and $\msfC$.
\end{abstract}

\maketitle

\section{Introduction}\label{sec: introduction}
Let $(W,S)$ be a Coxeter system
with generating set $S = \{ s_1, \dotsc, s_m \}$, let $\leq$ denote the Bruhat order on $W$, let $H = H(W)$ be the Iwahori--Hecke algebra corresponding to $W$, and let $\{ T_w \,|\, w \in W \}$ be the {\em natural basis} of $H$ as a $\zqq$-module. (See, e.g., \cite{BBCoxeter}.)
A second basis $\{ C'_w(q) \,|\, w \in W \}$ introduced by Kazhdan and Lusztig~\cite{KLRepCH}, sometimes rescaled as $\{ \wtc wq \,|\, w \in W \}$ with $\wtc wq \defeq q^{\ell(w)/2} C'_w(q)$, is important in many areas of mathematics. Products of these elements expand nonnegatively in the natural and Kazhdan--Lusztig bases, and have appeared in intersection homology~\cite{BBDFaisceaux}, \cite{SpringerQACI}, algorithmic and combinatorial description of Kazhdan--Lusztig basis elements themselves~\cite{BWHex}, \cite{Deodhar90}, Schubert varieties~\cite{BWHex}, total nonnegativity~\cite{GJImm}, \cite{SkanNNDCB}, \cite{StemImm}, \cite{StemConj}, 
trace evaluations~\cite{CHSSkanEKL}, \cite{CSkanTNNChar}, \cite{GreeneImm}, \cite{KLSBasesQMBIndSgn}, \cite{SkanHyperGC}, and chromatic symmetric functions~\cite{CHSSkanEKL}, \cite{SkanHyperGC}.

We will focus on Deodhar's result \cite[Prop.\,3.5]{Deodhar90}
concerning sequences $(s_{i_1},\dotsc,s_{i_k})$ of generators in $S$, products of the 
corresponding Kazhdan--Lusztig basis elements $\wtc{s_{i_j}\ntksp}q = T_e + T_{s_{i_j}}$ of and their natural expansions \begin{equation}\label{eq:Deoprod}
\wtc{s_{i_1}\ntksp}q \cdots \wtc{s_{i_k}\ntksp}q= \sum_{w \in W} a_w T_w.
\end{equation}

Deodhar described the coefficients $\{ a_w \,|\, w \in W \} \subset \mathbb Z[q]$ 
in terms of {\em subexpressions} of $(s_{i_1}, \dotsc, s_{i_k})$, sequences $\sigma = (\sigma_1, \dotsc, \sigma_k)$ with $\sigma_j \in \{ e, s_{i_j}\}$ for $j = 1,\dotsc, k$. Call index $j$ a {\em defect} of $\sigma$ if 
$\sigma_1 \cdots \sigma_{j-1}s_{i_j} < \sigma_1 \cdots \sigma_{j-1}$ and let $\dfct(\sigma)$ denote the number of defects of $\sigma$. Each coefficient on the right-hand side of (\ref{eq:Deoprod}) is given by
\begin{equation}\label{eq:Deodefectformula}
    a_w = \sum_\sigma q^{\dfct(\sigma)},
\end{equation}
where the sum is over all subexpressions
$\sigma$ of $(s_{i_1},\dotsc, s_{i_k})$
satisfying $\sigma_1 \cdots \sigma_k = w$. (Our terminology differs a bit from that of Deodhar.)

Billey and Warrington~\cite[Rmk.\,6]{BWHex}
observed that when $W$ and $H$ are
the symmetric group $\sn$ and type-$\msfA$ Iwahori--Hecke algebra
$\hanq$, the defect statistic has the following simple graphical interpretation.
Let $F = F_{s_{i_1}} \circ \cdots \circ F_{s_{i_k}}$ be the wiring diagram corresponding to $(s_{i_1}, \dotsc, s_{i_k})$, where $\circ$ denotes concatenation, and the factor wiring diagrams 
$F_{s_1}, \dotsc, F_{s_{n-1}}$
are the planar networks

\begin{equation}\label{eq:A generator wirings}
\begin{tikzpicture}[scale=.4,baseline=5]
  \node at (-1.4,3.5) {$\scriptstyle n$};
  \node at (1.4,3.5) {$\scriptstyle n$};
  \node at (-1.5,2.5) {$\scriptstyle n-1$};
  \node at (1.5,2.5) {$\scriptstyle n-1$};
  \node at (0,1.75) {$\vdots$};
  \node at (-1.4,0.5) {$\scriptstyle 3$};
  \node at (1.4,0.5) {$\scriptstyle 3$};
  \node at (1.4,-0.5) {$\scriptstyle 2$};
  \node at (-1.4,-0.5) {$\scriptstyle 2$};
  \node at (-1.4,-1.5) {$\scriptstyle 1$};
  \node at (1.4,-1.5) {$\scriptstyle 1$};
   \draw[-, thick]
  (-.5,3.5) -- (.5,3.5);
  \draw[-, thick]
  (-.5,2.5) -- (.5,2.5);
   \draw[-, thick]
  (-.5,0.5) -- (.5,0.5);
  \draw[-, thick]
    (-.5,-0.5) -- (0,-1) -- (0.5,-.5);
  \draw[-, thick]
  (-.5,-1.5) -- (0,-1) -- (0.5,-1.5);
\end{tikzpicture}
\ntksp, \quad 
\begin{tikzpicture}[scale=.4,baseline=5]
  \node at (-1.4,3.5) {$\scriptstyle n$};
  \node at (1.4,3.5) {$\scriptstyle n$};
  \node at (-1.5,2.5) {$\scriptstyle n-1$};
  \node at (1.5,2.5) {$\scriptstyle n-1$};
  \node at (0,1.75) {$\vdots$};
  \node at (-1.4,0.5) {$\scriptstyle 3$};
  \node at (1.4,0.5) {$\scriptstyle 3$};
  \node at (1.4,-0.5) {$\scriptstyle 2$};
  \node at (-1.4,-0.5) {$\scriptstyle 2$};
  \node at (-1.4,-1.5) {$\scriptstyle 1$};
  \node at (1.4,-1.5) {$\scriptstyle 1$};
     \draw[-, thick]
  (-.5,3.5) -- (.5,3.5);
  \draw[-, thick]
  (-.5,2.5) -- (.5,2.5);
   \draw[-, thick]
  (-.5,0.5) -- (0,0) -- (.5,0.5);
  \draw[-, thick]
    (-.5,-0.5) -- (0,0) -- (0.5,-.5);
  \draw[-, thick]
  (-.5,-1.5) -- (0.5,-1.5);
\end{tikzpicture}\ntksp, 
 \dotsc, 
\begin{tikzpicture}[scale=.4,baseline=5]
  \node at (-1.4,3.5) {$\scriptstyle n$};
  \node at (1.4,3.5) {$\scriptstyle n$};
  \node at (-1.5,2.5) {$\scriptstyle n-1$};
  \node at (1.5,2.5) {$\scriptstyle n-1$};
    \node at (0,1.75) {$\vdots$};
    \node at (-1.4,0.5) {$\scriptstyle 3$};
  \node at (1.4,0.5) {$\scriptstyle 3$};
  \node at (1.4,-0.5) {$\scriptstyle 2$};
  \node at (-1.4,-0.5) {$\scriptstyle 2$};
  \node at (-1.4,-1.5) {$\scriptstyle 1$};
  \node at (1.4,-1.5) {$\scriptstyle 1$};
   \draw[-, thick]
  (-.5,3.5) -- (0,3) -- (.5,3.5);
   \draw[-, thick]
  (-.5,2.5) -- (0,3) -- (.5,2.5);
   \draw[-, thick]
  (-.5,0.5) -- (.5,0.5);
  \draw[-, thick]
    (-.5,-0.5) -- (0.5,-.5);
  \draw[-, thick]
  (-.5,-1.5) -- (0.5,-1.5);
\end{tikzpicture}\ntksp,
\end{equation}
respectively.  Edges of $F$ are understood to be oriented from left to right, with $n$ {\em source}
vertices on the left,
$n$ {\em sink} vertices on the right,
and $k$ more {\em interior} vertices, one per factor. We can cover $F$ with $2^k$ different path families of the form $\pi = (\pi_1,\dotsc,\pi_n)$, if we allow two paths meeting at the interior vertex of a factor $F_{s_{i_j}}$ either to cross or not to cross there. Let $\Pi_w(F)$ be the subset of these path families 
having {\em type $w$}, i.e., for $i = 1,\dotsc, n$, path $\pi_i$ begins at source $i$ and terminates at sink $w_i$. Call index $j$ a {\em (type-$\msfA$) defect} of $\pi$ if the two paths meeting in $F_{s_{i_j}}$ cross an odd number of times in $F_{s_{i_1}} \circ \cdots \circ F_{s_{i_{j-1}}}$, and let $\dfct(\pi)$ be the number
of defects in $\pi$. Then each coefficient on the right-hand side of (\ref{eq:Deoprod}) is given by
\begin{equation}\label{eq:BWdefectformula}
    a_w = \sum_{\pi\in \Pi_w(F)} 
q^{\dfct(\pi)}.
\end{equation}

Clearwater and the third author~~\cite[Cor.\,5.3]{CSkanTNNChar} extended this type-$\msfA$ result to products of the form
\begin{equation}\label{eq:CSprod}
\wtc{v^{(1)}}q \cdots \wtc{v^{(k)}}q
= \sum_w a_w T_w 
\end{equation}
in $\hanq$,
where $v^{(1)}, \dotsc, v^{(k)}$
are maximal elements of parabolic subgroups of $\sn$, and 
each factor
\begin{equation*}
    \wtc{v^{(j)}}q \defeq \sum_{u \leq v^{(j)}} \ntksp T_u
\end{equation*}
belongs to the modified signless Kazhdan--Lusztig basis of $\hanq$.
Again we have (\ref{eq:BWdefectformula}), where 
$F$ has the form $F_{v^{(1)}} \circ \cdots \circ F_{v^{(k)}}$, 
with factors belonging to a class of planar networks generalizing those in
(\ref{eq:A generator wirings}) in that their interior vertices may have indegree and outdegree greater than $2$. More generally still, the formula (\ref{eq:BWdefectformula}) holds if permutations $v^{(1)}, \dotsc, v^{(k)}$ correspond to smooth type-$\msfA$ Schubert varieties, i.e., if they \avoidp.

We extend this result further to analogous products of elements
of the Kazhdan--Lusztig basis of the type-$\msfBC$ Iwahori--Hecke algebra 
$\hbnq$
in Section~\ref{s:main}.
To do so, we review necessary background material on the hyperoctahedral group in Section~\ref{sec: bn}, 
$\hbnq$ in Section~\ref{sec: Hecke algebra},
and type-$\msfBC$ planar networks in Section~\ref{sec: BCstarnets}. We discuss 
open problems in Section~\ref{s:open}.

\section{The hyperoctahedral group}
\label{sec: bn}

Recall that for a Coxeter group $W$ with generating set $S$
and an element $w \in W$, an expression $w = s_{i_1} \cdots s_{i_\ell}$ for $w$ as a product of generators is called {\em reduced}
if it is as short as possible, and $\ell = \ell(w)$ is called the {\em length} of $w$.  We define the {\em Bruhat order} on $W$ by
declaring $v \leq w$ if some (equivalently, every) reduced expression
for $w$ contains a subsequence which is a reduced expression for $v$.  By this definition we clearly have

\begin{equation}\label{eq:bruhatinverse}
    v \leq w \ \Longleftrightarrow\ v^{-1} \leq w^{-1}.
\end{equation}
For a subset $J \subseteq S$, let $\smash{W_J}$ denote the subgroup of $W$ generated by $J$, and call such a subgroup {\em parabolic}.
Each coset $w\smash{W_J}$ of $W_J$ has a unique Bruhat-minimal element and, if $W$ is finite, a unique Bruhat-maximal element~\cite[Cor.~2.4.5]{BBCoxeter}.
Let $\smash{\Wejm}$ denote the set of minimal coset representatives.
It is known \cite[Defn.~2.4.2, Lem.~2.4.3]{BBCoxeter} that we have

\begin{equation}\label{eq:wejmdefns}
    \Wejm = \{ w \in W \,|\, ws > w \text{ for all } s \in J \} 
    = \{ w \in W \,|\, wv \geq v \text{ for all } v \in W_J \} 
\end{equation}
and that there is a bijection~\cite[Prop.~2.4.4]{BBCoxeter}
\begin{equation}\label{eq:Jfactor}
    \Wejm \times W_J \;
    \overset{1-1}{\longleftrightarrow} \;
    W
    \end{equation}
    given by simple multiplication $(w,u) \mapsto wu$ and satisfying $\ell(w) + \ell(u) = \ell(wu)$.
The Bruhat order on $W$ induces a related partial order on the set $W / W_J$ of cosets $\{ wW_J \,|\, w \in \bn \}$. Specifically, we declare $vW_J \leq wW_J$ if elements of the cosets satisfy any of the three (equivalent) inequalities in the Bruhat order on $W$~\cite[Lem.\,2.2]{DougInv}.
  \begin{enumerate}[(i)]
      \item The minimal element of $vW_J$ is less than
      or equal to
the minimal element of $wW_J$.
      \item The maximal element of $vW_J$ is less than
      or equal to
the maximal element of $wW_J$, if these exist.
      \item At least one element of $vW_J$ is less than
      or equal to
at least one element of $wW_J$.
  \end{enumerate}

  We call this poset the {\em Bruhat order on $W / W_J$}.

The {\em hyperoctahedral group}
$\,\bn$ is the Coxeter group of type $\mathsf B_n = \mathsf C_n$, with generators $S=\{s_0,\dotsc,s_{n-1}\}$, and relations

\begin{equation}\label{eq:bnpresentation}
  \begin{alignedat}2
    {s_i}^2 &= e &\quad &\text{for $i = 0, \dotsc, n-1$,}\\
    s_0s_1s_0s_1 &= s_1s_0s_1s_0, &\quad & \\
    s_is_j &= s_js_i &\quad &\text{for $i,j \geq 0$ and $|i-j| \geq 2$,}\\
    s_is_js_i &= s_js_is_j &\quad &\text{for $i,j \geq 1$ and $|i-j| = 1$.}
  \end{alignedat}
\end{equation}

The parabolic subgroup of $\bn$ generated by $\{s_1,\dotsc,s_{n-1}\}$ is the {\em symmetric group} $\sn$, the Coxeter group of type $\msfA_{n-1}$.

Like elements of the symmetric group, elements of the hyperoctahedral group naturally correspond to permutations of
letters belonging to a certain alphabet.

Specifically, we define the alphabet
\begin{equation*}
  [\ol n, n] \defeq \{ -n, \dotsc, n \} \ssm \{ 0 \}
  \end{equation*}
with notation $\ol a \defeq -a$ for all $a \in [\ol n,n]$,
and we consider permutations $w_{\ol n} \cdots w_{\ol 1} w_1 \cdots w_n$ of $[\ol n, n]$ which satisfy $w_{\ol i} = \ol{w_i}$ for all $i \in [\ol n, n]$. We define a correspondence between $\bn$ and such permutations via the (left) action of $\bn$ on these permutations defined by
\begin{enumerate}
    \item $s_0$ swaps the letters in positions $\ol1$, $1$,
    \item $s_i$ ($i = 1,\dotsc,n-1$) simultaneously swaps 
    the letters in positions $i$, $i+1$ with each other,
    and
    the letters in positions $\ol i$, $\ol{i+1}$ with each other.
\end{enumerate}
Then for $w = s_{i_1} \cdots s_{i_r} \in \bn$,
we define the {\em (long) one-line notation} of $w$ to be the permutation

\begin{equation}\label{eq:s2nrearrange}
  w_{\ol n} \cdots w_{\ol 1} w_1 \cdots w_n 
= s_{i_1}(s_{i_2}( \cdots (s_{i_r}( \ol n \cdots \ol 1 1 \cdots n)) \cdots )).
\end{equation}

For example, when $n=4$, the element $s_0s_1 \in \mfb4$ has long one-line notation
\begin{equation}\label{eq:long1lineex}
    s_0( s_1(\ol4 \ol3 \ol2 \ol1 1234))
    = s_0(\ol4 \ol3 \ol1 \ol2 2 1 3 4)
    = \ol4 \ol3 \ol1 2 \ol2 1 3 4.
  \end{equation}

(By our definition, the right action of $s_0$ swaps the letters $1, \ol 1$,
wherever they are, while the right action of $s_i$ for $i = 1,\dotsc,n-1$ swaps
the letters $i$, $i+1$, wherever they are, while simultaneously swapping the letters $\ol i$, $\ol{i+1}$,
wherever they are.) It follows that $w_i^{-1}$ is the index $j$ satisfying $w_j = i$.

The condition $w_{\ol i} = \ol{w_i}$
implies that each element (\ref{eq:s2nrearrange}) is completely determined by the subword $w_1 \cdots w_n$, called the \textit{short one-line notation} of $w$.

The set of short one-line notations of elements of $\bn$ is the set of {\em signed permutations} of $[1,n]$:
words $w_1 \cdots w_n$ with letters in the alphabet $[\ol n,n]$ 
having no repeated absolute values.

The correspondence between $\bn$ and signed permutations is given by the left action of $\bn$ on signed permutations defined by
\begin{enumerate}
    \item $s_0$ negates the letter in position $1$, and
    \item $s_i$ ($i = 1,\dotsc,n-1$) swaps the letters in positions $i$, $i+1$ with each other.
\end{enumerate}
For example the short one-line notation $\ol2134$ of the element $s_0s_1 \in \mfb{4}$
(\ref{eq:long1lineex}) can be obtained by its action on $1234$: 

\begin{equation}\label{eq:short1lineex}
   s_0( s_1(1234))
    = s_0(2 1 3 4)
    = 
\ol2 1 3 4.
\end{equation}
(By our definition, the right action of $s_0$ on signed permutations negates the letter $1$ or $\ol1$, wherever it is, while the right action of $s_i$ for $i = 1,\dotsc,n-1$ swaps
the letters having absolute values $i$, $i+1$, wherever they are, while leaving any negation signs in their original positions.)  

Some properties of elements of $\bn$ can be stated in terms of {\em pattern avoidance}.  Given a word $u = u_1 \cdots u_k$ with distinct letters, we say that a subword $w_{i_1} \cdots w_{i_k}$ of $w_{\ol n} \cdots w_{\ol1} w_1 \cdots w_n$ {\em matches the pattern $u$} if we have
\begin{equation*}
    w_{i_h} < w_{i_j}
    \text{ if and only if }
    u_h < u_j
\end{equation*}
    for all $1 \leq h < j \leq k$.
    We say that $w$ {\em avoids the pattern $u$} if no subword of $w$ matches the pattern $u$.

We define an \emph{inversion} of any word $w_1 \cdots w_n$ to be a pair $(i,j)$ of indices satisfying $i < j$ and $w_i > w_j$. Let $\inv(w)$ denote the number of inversions in $w$. For $w \in \sn$, it is well known that we have $\inv(w_1 \cdots w_n) = \ell(w)$. For $w \in \bn$ we can similarly compute $\ell(w)$ by inspecting the short one-line notation of $w$.  (See, e.g., \cite[Prop.\,8.1.1,  Eqn.\,(8.2)]{BBCoxeter}.)
\begin{lem}\label{l:BClength}
  Let $w \in \bn$ have short one-line notation $w_1 \cdots w_n$.
  Then its length satisfies
  \begin{equation*}
    \ell(w) = \inv(w_1 \cdots w_n) + \sumsb{i\\w_i < 0} |w_i|.
  \end{equation*}
\end{lem}
\noindent We can also compute $\ell(w)$ by counting certain inversions in the long one-line notation of $w$.
\begin{lem}\label{l:longinversions}
The length of $w \in \bn$ equals the number of pairs $(i,j)$ with $|i| \leq j$ and $j$ appearing earlier than $i$ in
    $w_{\ol n} \cdots w_{\ol1} w_1 \cdots w_n$.
    \end{lem}
    \begin{proof}
Let $\ell'$ be the proposed formula for length, and assume that $\ell'$ agrees with the true length $\ell$ for hyperoctahedral groups $\mathfrak B_1,\dotsc, \mathfrak B_{n-1}$. Consider $w  = w_{\ol n} \cdots w_{\ol1} w_1 \cdots w_n \in \bn$ satisfying $w_k = n$, $w_{\ol k} = \ol n$ for some 
$k \in [\ol n,n]$, and define $v = v_{\ol{n-1}} \cdots v_{\ol1} v_1 \cdots v_{n-1} \in \mathfrak B_{n-1}$ 
by removing letters $n$, $\ol n$ from $w$. By Lemma~\ref{l:BClength} we have 

\begin{equation}\label{eq:lengthinduction}
    \ell(w) = \begin{cases}
        \ell(v) + n-k &\text{if $k >0$},\\
       \ell(v) + |k|-1 + n &\text{if $k < 0$},
       \end{cases}
\end{equation}
since $n$ forms inversions with $w_{k+1}, \dotsc, w_n$ if it appears in position $k > 0$, and $\ol n$ forms inversions with $w_1 \cdots w_{|k|-1}$ if it appears in position $|k|$ with $k < 0$. On the other hand, by the definition of $\ell'$ and induction we have
\begin{equation*}
\begin{aligned}
       \ell'(w) &= 
       \ell'(v) + \# \text{ letters following $n$ in $w_{\ol n} \cdots w_{\ol1} w_1 \cdots w_n$} \\
       &= \ell(v) + \# \text{ letters following $n$ in $w_{\ol n} \cdots w_{\ol1} w_1 \cdots w_n$}.
    \end{aligned}
\end{equation*}

By (\ref{eq:lengthinduction}), this is $\ell(w)$.
 \end{proof}

Certain elements of the hyperoctahedral group are most easily defined in terms of subintervals of $[\ol n, n]$, where we declare any subset 
\begin{equation*}
    [a,b] \defeq \{ a, \dotsc, b \} \ssm \{0 \}
\end{equation*}
 of $[\ol n,n]$ 
 to be 
 an {\em interval}, even if $a < 0 < b$.
Roughly speaking, we define a {\em reversal} of $\bn$ to be an element $s_{[a,b]}$ obtained from the identity by reversing letters $[a,b]$ in positions $[a,b]$, and ensuring that the resulting permutation belongs to $\bn$.
To be precise, we describe such elements using three cases: $a=b$, $a = \ol b$ ($b > 0$), and $0 < a < b$. When $a=b$, we have the trivial reversal $s_{[b,b]} = s_\emptyset = e$. 
For $b>0$, the reversal $s_{[\ol b,b]}$ is the element having one-line notation
\begin{equation}
 \ol n \cdots \ol{(b+1)} \cdot b \cdots 1 \cdot \ol 1 \cdots \ol b \cdot (b+1) \cdots n,
\end{equation}
and equal to the product of generators $s_0 (s_1 s_0 s_1)(s_2s_1s_0s_1s_2) \cdots (s_{b-1} \dotsc s_1 s_0 s_1 \dotsc s_{b-1})$. For $0 < a < b$ the reversal $s_{[a,b]}$ is the element having 
one-line notation
\begin{equation*}
\ol n \cdots \ol{(b+1)} \cdot \ol a \cdots \ol b \cdot \ol{(a-1)} \cdots \ol 1
\cdot 
1 \cdots (a-1) \cdot b \cdots a \cdot (b+1) \cdots n,
\end{equation*}
and equal to the product of generators $s_a (s_{a+1}s_a)(s_{a+2}s_{a+1}s_a) \cdots ( s_{b-1} \cdots s_a)$. Each reversal $s_{[a,b]}$ is the unique element of maximum length in a parabolic subgroup of $\bn$, generated by
\begin{equation}\label{eq:jabdef}
    J_{[a,b]} \defeq \begin{cases}
        \{ s_0, \dotsc, s_{b-1} \} &\text{if $a = \ol b$ ($b > 0$)},\\
        \{ s_a, \dotsc, s_{b-1} \} &\text{if $0 < a < b$}.
        \end{cases}
\end{equation}

The first equality in (\ref{eq:wejmdefns}) implies that when $J = J_{[a,b]}$, the subset $\Wejm$ of $\bn$ may be characterized in terms
of long one-line notation.  Specifically, we have the following.

\begin{lem}\label{l:minreps}
For $J = J_{[a,b]}$, each minimum-length coset representative $w \in \Wejm$ satisfies

\begin{enumerate}[(i)]
    \item 
    $w_{\ol b}^{-1} < \cdots < w_{\ol a}^{-1}$ and $w_a^{-1} < \cdots < w_b^{-1}$ for $a > 0$, 
\item $w_{\ol b}^{-1} < \cdots < w_{\ol1}^{-1} < w_1^{-1} < \cdots < w_b^{-1}$ for $a = \ol b$.
\end{enumerate}
\end{lem}

\begin{proof}
By 
(\ref{eq:bruhatinverse}) and
the action preceding (\ref{eq:s2nrearrange}), we have the equivalences

\begin{equation*}
    ws_i > w
    \ \Longleftrightarrow\
    s_i w^{-1} > w^{-1}
    \ \Longleftrightarrow\
\begin{cases}
    w_i^{-1}<w_{i+1}^{-1}
    &\text{if $1 \leq i \leq n-1$},\\
    w_{\ol1}^{-1} < w_1^{-1}
    &\text{if $i = 0$}.
\end{cases}
\end{equation*}
Thus by (\ref{eq:wejmdefns}), (\ref{eq:jabdef}), the combination of conditions $(i)$, $(ii)$
is equivalent to the minimality of $w$ in 
$wW_J$.
\end{proof}

\section{The Hecke algebra of the hyperoctahedral group}\label{sec: Hecke algebra}

Given Coxeter group $W$ with generator set $S$, define
the {\em Hecke algebra} $H = H(W)$ of $W$ to be
the $\zqq$-span of $\{ T_w \,|\, w \in W \}$
with unit $T_e$ and multiplication defined by
\begin{equation}\label{eq:genlhecke}
  T_wT_s =
  \begin{cases}
    qT_{ws} + (q-1)T_w &\text{if $ws < w$},\\
    T_{ws} &\text{if $ws > w$},
  \end{cases}
\end{equation}
where $s \in S$ and $w \in W$. This formula guarantees that for $w \in W$ and
any reduced expression $s_{i_1} \cdots s_{i_\ell}$ for $w$, we have $T_w = \smash{T_{s_{i_1}} \cdots T_{s_{i_\ell}}}$.
Consequently, for $v, w \in W$ we have the implication

\begin{equation}\label{eq:tvtwtvw}
    \ell(v) + \ell(w) = \ell(vw) \quad \Rightarrow \quad T_vT_w = T_{vw}.
\end{equation}
Call $\{ T_w \,|\, w \in W \}$ the {\em natural basis} of $H$.

A second basis~\cite{KLRepCH} 
of $H$ is the (rescaled)
{\em Kazhdan--Lusztig basis} $\{ \wtc wq \,|\, w \in W \}$,
related to the natural basis by
\begin{equation}\label{eq:wtcdef}
  \wtc wq = \sum_{v \leq w} P_{v,w}(q) T_v,
\end{equation}
where
$\{ P_{v,w}(q) \,|\, v, w \in W \} \subseteq \mathbb Z[q]$ are the
{\em Kazhdan--Lusztig polynomials},
whose recursive definition appears in ~\cite{KLRepCH}. Coefficients of these polynomials may be interpreted in terms of
intersection cohomology $\icoh^*(\schub w)$~\cite{KLSchub}, where $\schub w$ is a certain {\em Schubert variety} indexed by $w$. (See, e.g., \cite{BilleyLak}.) In particular, when $\schub w$ is rationally smooth, all polynomials $\{P_{v,w}(q) \,|\, v \leq w\}$ are identically $1$ \cite[Thm.\,A.2]{KLRepCH}.

For each subset $J$ of generators of $W$, one forms a natural $\zqq$-submodule 
of $H$ by taking the span of sums

\begin{equation}\label{eq:cosetsum}
T_{uW_J} \defeq \ntksp \sum_{v \in uW_J} \ntksp T_v
\end{equation}
of natural basis elements of $H$ with each sum corresponding to a coset of $W_J$. This submodule
\begin{equation}\label{eq:hbjqdef}
H_J
\defeq \spn_{\zqq} \{ T_{uW_J} \,|\, u \in \Wejm \}
\end{equation}
in fact forms a left $H$-module. A nice formula for the action of $H$ on $H_J$ was given by Douglass~\cite[Prop.~2.3]{DougInv}. 
\begin{prop}\label{p:douglass}
    Let $s$ be a generator of $W$, and let $J$ 
    be some subset of the generators. 
    Then for all $w \in  W$, we have
\begin{equation} T_s T_{w W_J} = \left\{ \begin{array}{cc}
    q T_{s w W_J} + (q-1)T_{w W_J} & \text{if } swW_J < wW_J, \\
    T_{s w W_J} & \text{if } swW_J > wW_J ,\\
    q T_{w W_J} & \text{if } swW_J = wW_J .\end{array} \right.
    \end{equation}
\end{prop}

\begin{cor}\label{c:doug}
    If $v \in W_J$, then
    $T_{v}T_{W_J} = q^{\ell(v)} T_{W_J}$.
   \end{cor}

Let us consider the type-$\msfBC$ Hecke algebra $\hbnq$, generated by
$T_{s_0}, \dotsc, T_{s_{n-1}}$ subject to relations
\begin{equation}\label{eq:hbnqpresentation}
  \begin{alignedat}2
    T_{s_i}^2 &= (q-1)T_{s_i} + qT_e &\quad &\text{for $i = 0, \dotsc, n-1$,}\\
    T_{s_0}T_{s_1}T_{s_0}T_{s_1} &= T_{s_1}T_{s_0}T_{s_1}T_{s_0}, &\quad & \\
    T_{s_i}T_{s_j} &= T_{s_j}T_{s_i} &\quad &\text{for $i,j \geq 0$ and $|i-j| \geq 2$,}\\
    T_{s_i}T_{s_j}T_{s_i} &= T_{s_j}T_{s_i}T_{s_j} &\quad &\text{for $i,j \geq 1$ and $|i-j| = 1$,}
  \end{alignedat}
\end{equation}
a subset $J = J_{[a,b]}$ of generators of $\bn$ for some $a < b$, and the submodule $\hbjq \defeq H_J$ of $\hbnq$. By \cite[\S 13.3.7]{BilleyLak},
rational smoothness of Schubert varieties of types $\msfA$, $\msfBB$, $\msfC$ 
can be characterized in terms of pattern avoidance. In particular, when $w \in \bn$ \avoidsp, we have the stronger property
that type-$\msfBB$ and $\msfC$ Schubert varieties indexed by $w$ are both smooth. Since a reversal in $\bn$ avoids both of these patterns, we have the following.

\begin{prop}\label{p:wtcreversal}
    For each reversal $s_{[a,b]} \in \bn$,
    we have
\begin{equation*}
\wtc{s_{[a,b]}}q = \sum_{v\leq s_{[a,b]}} T_v.     
    \end{equation*}
\end{prop}

Each Kazhdan--Lusztig basis element 
indexed by a reversal $s_{[a,b]}$ is itself a coset sum (\ref{eq:cosetsum}) for the parabolic subgroup generated by $J = J_{[a.b]}$:
\begin{equation}
T_{W_J} = 
    \wtc{s_{[a,b]}}q.
\end{equation}
By (\ref{eq:Jfactor}), other defining basis elements (\ref{eq:hbjqdef}) of $\hbjq$ can be written as
\begin{equation}\label{eq:TuWJ}
    T_{uW_J} = T_u \wtc{s_{[a,b]}} q,
\end{equation}
for $u \in \Wejm$.
Thus by Proposition~\ref{p:douglass}, each product of the form
    \begin{equation}\label{eq:klprod}
        \wtc{s_{[a_1,b_1]}}q \cdots \wtc{s_{[a_{k},b_{k}]}}q 
    \end{equation}
belongs to $H^{\msfBC}_{J_{[a_k,b_k]}}(q)$ and we can expand it in the defining basis of this module as follows.

\begin{prop}\label{p:klprodcosetexpansion} 
Suppose that the product of the first $k-1$ factors of (\ref{eq:klprod})
expands in the natural basis of $\hbnq$
as

\begin{equation}\label{eq:klprodkminus1}
        \wtc{s_{[a_1,b_1]}}q \cdots \wtc{s_{[a_{k-1},b_{k-1}]}}q = \sum_{v \in \bn} c_v T_v,
    \end{equation}
    for some polynomials $\{ c_v = c_v(q)\,|\, v \in \bn \}$ in $\mathbb Z[q]$,
    and define $J = J_{[a_k,b_k]}$. Then the full product (\ref{eq:klprod})
expands in the defining basis of $\hbjq$ as
\begin{equation}\label{eq:klprodcosetexpansion}
\sum_{w \in \Wejm} \ntksp \bigg( \sum_{u \in  W_{J}} \ntnsp q^{\ell(u)} c_{wu} \bigg) T_w \wtc {s_{[a_k,b_k]}} q.
\end{equation}
\end{prop}
\begin{proof}
Using
(\ref{eq:Jfactor}) to factor each element $v$ on the right-hand side of 
(\ref{eq:klprodkminus1}) as $v = wu$ with $w \in \Wejm$, $u \in W_J$,
we may express the product (\ref{eq:klprod}) as

\begin{equation*}  
\sum_{w \in \Wejm} \ntksp \bigg( \sum_{u \in  W_{J}} \ntksp c_{wu} T_w T_u \bigg) \wtc 
 {s_{[a_k,b_k]}} q.
\end{equation*}

Then by Corollary \ref{c:doug}, 
we have $T_u \wtc{s_{[a_k,b_k]}}q
=
q^{\ell(u)}\wtc{s_{[a_k,b_k]}}q$ and the claimed formula.
\end{proof}

\section{Type-$\msfBC$ star networks}\label{sec: BCstarnets}

To graphically represent products of 
the form (\ref{eq:klprod}),
we extend the idea of type-$\msfBC$ wiring
diagrams to include planar networks in which interior vertices may have indegrees and outdegrees greater than 2. In particular, we associate to each reversal $s_{[a,b]} \in \bn$ a {\em type-$\msfBC$ simple star network} $\smash{F_{[a,b]}}$ having $2n$ source vertices on the left and $2n$ sink vertices on the right, both labeled $\ol n, \dotsc, n$ from bottom to top. For the three cases $a=b$, $a = \ol b$ ($b > 0$), $0 < a < b$ of reversals,
we include edges and $0$, $1$, or $2$ additional
interior vertices as follows.

\begin{enumerate}

\item When $a=b$, 
$\smash{F_{[b,b]} = F_\emptyset}$ has a directed edge from source $i$ to sink $i$, for $i = \ol n,\dotsc,n$.
\item When $a = \ol b$ ($b > 0$), $F_{[a,b]}$ has one interior vertex. For $i = a,\dotsc, b$, a directed edge begins at source $i$ and terminates at the interior vertex, and another directed edge begins at the interior vertex and terminates at sink $i$.
\item When $0 < a < b$, $F_{[a,b]}$ has two interior vertices, which we call {\em upper} and {\em lower}. For $i = a, \dotsc, b$, a directed edge begins at source $i$ and terminates at the upper interior vertex, and another directed edge begins at the upper interior vertex and terminates at sink $i$. For $i = \ol b, \dotsc, \ol a$, a directed edge begins at source $i$ and terminates at the lower interior vertex, and another directed edge begins at the lower interior vertex and terminates at sink $i$. 
\end{enumerate}

For economy of figures, we omit vertices and left-to-right orientations of edges. For example, the seven type-$\msfBC$ simple star networks

\begin{equation}\label{eq:csimplestars}        
\begin{tikzpicture}[scale=.5,baseline=0]
  \node at (-.9,2.5) {$\scriptstyle 3$};
  \node at (-.9,1.5) {$\scriptstyle 2$};
  \node at (-.9,0.5) {$\scriptstyle 1$};
  \node at (0.9,2.5) {$\scriptstyle 3$};
  \node at (0.9,1.5) {$\scriptstyle 2$};
  \node at (0.9,0.5) {$\scriptstyle 1$};
  \node at (-.9,-2.5) {$\scriptstyle {\ol 3}$};
  \node at (-.9,-1.5) {$\scriptstyle {\ol 2}$};
  \node at (-.9,-0.5) {$\scriptstyle {\ol 1}$};
  \node at (0.9,-2.5) {$\scriptstyle {\ol 3}$};
  \node at (0.9,-1.5) {$\scriptstyle {\ol 2}$};
  \node at (0.9,-0.5) {$\scriptstyle {\ol 1}$};
\draw[-, very thick] (-.5,2.5) -- (.5,2.5);
\draw[-, very thick] (-.5,1.5) -- (.5,1.5);
\draw[-, very thick] (-.5,0.5) -- (.5,0.5);
\draw[-, very thick] (-.5,-0.5) -- (.5,-0.5);
\draw[-, very thick] (-.5,-1.5) -- (.5,-1.5);
\draw[-, very thick] (-.5,-2.5) -- (.5,-2.5);
\node at (0, -3.8) {$F_\emptyset$};
\end{tikzpicture}
\ntnsp,\quad
\begin{tikzpicture}[scale=.5,baseline=0]
  \node at (-.9,2.5) {$\scriptstyle 3$};
  \node at (-.9,1.5) {$\scriptstyle 2$};
  \node at (-.9,0.5) {$\scriptstyle 1$};
  \node at (0.9,2.5) {$\scriptstyle 3$};
  \node at (0.9,1.5) {$\scriptstyle 2$};
  \node at (0.9,0.5) {$\scriptstyle 1$};
  \node at (-.9,-2.5) {$\scriptstyle {\ol 3}$};
  \node at (-.9,-1.5) {$\scriptstyle {\ol 2}$};
  \node at (-.9,-0.5) {$\scriptstyle {\ol 1}$};
  \node at (0.9,-2.5) {$\scriptstyle {\ol 3}$};
  \node at (0.9,-1.5) {$\scriptstyle {\ol 2}$};
  \node at (0.9,-0.5) {$\scriptstyle {\ol 1}$};
\draw[-, very thick] (-.5,2.5) -- (.5,2.5);
\draw[-, very thick] (-.5,1.5) -- (.5,1.5);
\draw[-, very thick] (-.5,0.5) -- (.5,-0.5);
\draw[-, very thick] (-.5,-0.5) -- (.5,0.5);
\draw[-, very thick] (-.5,-1.5) -- (.5,-1.5);
\draw[-, very thick] (-.5,-2.5) -- (.5,-2.5);
\node at (0, -3.8) {$F_{[\ol1,1]}$};
\end{tikzpicture}
\ntnsp,\quad
\begin{tikzpicture}[scale=.5,baseline=0]
  \node at (-.9,2.5) {$\scriptstyle 3$};
  \node at (-.9,1.5) {$\scriptstyle 2$};
  \node at (-.9,0.5) {$\scriptstyle 1$};
  \node at (0.9,2.5) {$\scriptstyle 3$};
  \node at (0.9,1.5) {$\scriptstyle 2$};
  \node at (0.9,0.5) {$\scriptstyle 1$};
  \node at (-.9,-2.5) {$\scriptstyle {\ol 3}$};
  \node at (-.9,-1.5) {$\scriptstyle {\ol 2}$};
  \node at (-.9,-0.5) {$\scriptstyle {\ol 1}$};
  \node at (0.9,-2.5) {$\scriptstyle {\ol 3}$};
  \node at (0.9,-1.5) {$\scriptstyle {\ol 2}$};
  \node at (0.9,-0.5) {$\scriptstyle {\ol 1}$};
\draw[-, very thick] (-.5,2.5) -- (.5,2.5);
\draw[-, very thick] (-.5,1.5) -- (.5,-1.5);
\draw[-, very thick] (-.5,0.5) -- (.5,-0.5);
\draw[-, very thick] (-.5,-0.5) -- (.5,0.5);
\draw[-, very thick] (-.5,-1.5) -- (.5,1.5);
\draw[-, very thick] (-.5,-2.5) -- (.5,-2.5);
\node at (0, -3.8) {$F_{[\ol2,2]}$};
\end{tikzpicture}
\ntnsp,\quad
\begin{tikzpicture}[scale=.5,baseline=0]
  \node at (-.9,2.5) {$\scriptstyle 3$};
  \node at (-.9,1.5) {$\scriptstyle 2$};
  \node at (-.9,0.5) {$\scriptstyle 1$};
  \node at (0.9,2.5) {$\scriptstyle 3$};
  \node at (0.9,1.5) {$\scriptstyle 2$};
  \node at (0.9,0.5) {$\scriptstyle 1$};
  \node at (-.9,-2.5) {$\scriptstyle {\ol 3}$};
  \node at (-.9,-1.5) {$\scriptstyle {\ol 2}$};
  \node at (-.9,-0.5) {$\scriptstyle {\ol 1}$};
  \node at (0.9,-2.5) {$\scriptstyle {\ol 3}$};
  \node at (0.9,-1.5) {$\scriptstyle {\ol 2}$};
  \node at (0.9,-0.5) {$\scriptstyle {\ol 1}$};
\draw[-, very thick] (-.5,2.5) -- (.5,-2.5);
\draw[-, very thick] (-.5,1.5) -- (.5,-1.5);
\draw[-, very thick] (-.5,0.5) -- (.5,-0.5);
\draw[-, very thick] (-.5,-0.5) -- (.5,0.5);
\draw[-, very thick] (-.5,-1.5) -- (.5,1.5);
\draw[-, very thick] (-.5,-2.5) -- (.5,2.5);
\node at (0, -3.8) {$F_{[\ol3,3]}$};
\end{tikzpicture}
\ntnsp,\quad
\begin{tikzpicture}[scale=.5,baseline=0]
  \node at (-.9,2.5) {$\scriptstyle 3$};
  \node at (-.9,1.5) {$\scriptstyle 2$};
  \node at (-.9,0.5) {$\scriptstyle 1$};
  \node at (0.9,2.5) {$\scriptstyle 3$};
  \node at (0.9,1.5) {$\scriptstyle 2$};
  \node at (0.9,0.5) {$\scriptstyle 1$};
  \node at (-.9,-2.5) {$\scriptstyle {\ol 3}$};
  \node at (-.9,-1.5) {$\scriptstyle {\ol 2}$};
  \node at (-.9,-0.5) {$\scriptstyle {\ol 1}$};
  \node at (0.9,-2.5) {$\scriptstyle {\ol 3}$};
  \node at (0.9,-1.5) {$\scriptstyle {\ol 2}$};
  \node at (0.9,-0.5) {$\scriptstyle {\ol 1}$};
\draw[-, very thick] (-.5,2.5) -- (.5,2.5);
\draw[-, very thick] (-.5,1.5) -- (.5,0.5);
\draw[-, very thick] (-.5,0.5) -- (.5,1.5);
\draw[-, very thick] (-.5,-0.5) -- (.5,-1.5);
\draw[-, very thick] (-.5,-1.5) -- (.5,-0.5);
\draw[-, very thick] (-.5,-2.5) -- (.5,-2.5);
\node at (0, -3.8) {$F_{[1,2]}$};
\end{tikzpicture}
\ntnsp,\quad
\begin{tikzpicture}[scale=.5,baseline=0]
  \node at (-.9,2.5) {$\scriptstyle 3$};
  \node at (-.9,1.5) {$\scriptstyle 2$};
  \node at (-.9,0.5) {$\scriptstyle 1$};
  \node at (0.9,2.5) {$\scriptstyle 3$};
  \node at (0.9,1.5) {$\scriptstyle 2$};
  \node at (0.9,0.5) {$\scriptstyle 1$};
  \node at (-.9,-2.5) {$\scriptstyle {\ol 3}$};
  \node at (-.9,-1.5) {$\scriptstyle {\ol 2}$};
  \node at (-.9,-0.5) {$\scriptstyle {\ol 1}$};
  \node at (0.9,-2.5) {$\scriptstyle {\ol 3}$};
  \node at (0.9,-1.5) {$\scriptstyle {\ol 2}$};
  \node at (0.9,-0.5) {$\scriptstyle {\ol 1}$};
\draw[-, very thick] (-.5,2.5) -- (.5,1.5);
\draw[-, very thick] (-.5,1.5) -- (.5,2.5);
\draw[-, very thick] (-.5,0.5) -- (.5,0.5);
\draw[-, very thick] (-.5,-0.5) -- (.5,-0.5);
\draw[-, very thick] (-.5,-1.5) -- (.5,-2.5);
\draw[-, very thick] (-.5,-2.5) -- (.5,-1.5);
\node at (0, -3.8) {$F_{[2,3]}$};
\end{tikzpicture}
\ntnsp,\quad
\begin{tikzpicture}[scale=.5,baseline=0]
  \node at (-.9,2.5) {$\scriptstyle 3$};
  \node at (-.9,1.5) {$\scriptstyle 2$};
  \node at (-.9,0.5) {$\scriptstyle 1$};
  \node at (0.9,2.5) {$\scriptstyle 3$};
  \node at (0.9,1.5) {$\scriptstyle 2$};
  \node at (0.9,0.5) {$\scriptstyle 1$};
  \node at (-.9,-2.5) {$\scriptstyle {\ol 3}$};
  \node at (-.9,-1.5) {$\scriptstyle {\ol 2}$};
  \node at (-.9,-0.5) {$\scriptstyle {\ol 1}$};
  \node at (0.9,-2.5) {$\scriptstyle {\ol 3}$};
  \node at (0.9,-1.5) {$\scriptstyle {\ol 2}$};
  \node at (0.9,-0.5) {$\scriptstyle {\ol 1}$};
\draw[-, very thick] (-.5,2.5) -- (.5,0.5);
\draw[-, very thick] (-.5,1.5) -- (.5,1.5);
\draw[-, very thick] (-.5,0.5) -- (.5,2.5);
\draw[-, very thick] (-.5,-0.5) -- (.5,-2.5);
\draw[-, very thick] (-.5,-1.5) -- (.5,-1.5);
\draw[-, very thick] (-.5,-2.5) -- (.5,-0.5);
\node at (0, -3.8) {$F_{[1,3]}$};
\end{tikzpicture}
\ntnsp,\quad
\ntksp
\end{equation}
correspond to the reversals
$s_\emptyset = e$, $s_{[\ol1,1]}$, $s_{[\ol2,2]}$, $s_{[\ol3,3]}$, $s_{[1,2]}$, $s_{[2,3]}$, $s_{[1,3]}$ in $\mfb3$.
Traditional {\em type-$\msfBC$ wiring diagrams} are formed by concatenating networks 
$F_{[\ol1,1]}, F_{[1,2]}, \dotsc, F_{[n-1,n]}$
indexed by intervals of cardinality $2$. More generally, define {\em type-$\msfBC$ star networks} to be the set $\net{BC}n$ of all concatenations
\begin{equation}\label{eq:BCconcat}
F_{[a_1,b_1]} \circ \cdots \circ F_{[a_k,b_k]}
\end{equation}
of type-$\msfBC$ simple star networks, 
where 
$F_{[a_i,b_i]} \circ F_{[a_{i+1},b_{i+1}]}$ is formed, informally, by identifying all sink vertices of $F_{[a_i,b_i]}$ with the corresponding source vertices of $F_{[a_{i+1},b_{i+1}]}$.  More precisely,
for all $j$ we replace the unique edge $(x, \snk\;j)$ of $F_{[a_i,b_i]}$ and the unique
edge $(\src\;j, y)$ of $F_{[a_{i+1},b_{i+1}]}$
by a single edge $(x,y)$
in $F_{[a_i,b_i]} \circ F_{[a_{i+1},b_{i+1}]}$. Again for economy, we omit vertices and edge orientations in the resulting directed multigraph. For instance, the type-$\msfBC$ star network $F_{[1,3]} \circ F_{[2,3]} \circ F_{[1,2]} \circ F_{[\ol1,1]} \in \net{BC}3$, drawn once economically without vertices and edge orientations and once with all details, is

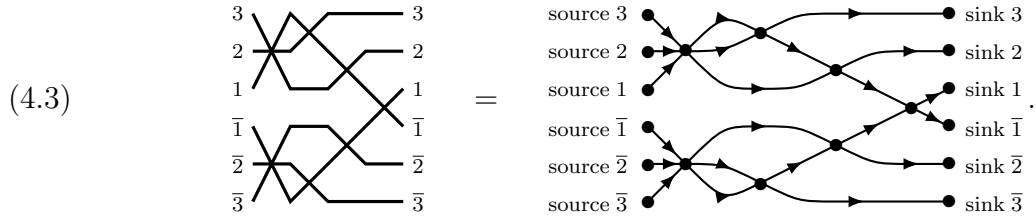
\begin{equation}\label{eq:concatex}
\begin{tikzpicture}[scale=.5,baseline=0]
  \node at (-.9,2.5) {$\scriptstyle 3$};
  \node at (-.9,1.5) {$\scriptstyle 2$};
  \node at (-.9,0.5) {$\scriptstyle 1$};
  \node at (3.9,2.5) {$\scriptstyle 3$};
  \node at (3.9,1.5) {$\scriptstyle 2$};
  \node at (3.9,0.5) {$\scriptstyle 1$};
  \node at (-.9,-2.5) {$\scriptstyle {\ol 3}$};
  \node at (-.9,-1.5) {$\scriptstyle {\ol 2}$};
  \node at (-.9,-0.5) {$\scriptstyle {\ol 1}$};
  \node at (3.9,-2.5) {$\scriptstyle {\ol 3}$};
  \node at (3.9,-1.5) {$\scriptstyle {\ol 2}$};
  \node at (3.9,-0.5) {$\scriptstyle {\ol 1}$};
  \draw[-, very thick]
  (-.5,2.5) -- (0,1.5) -- (.5,2.5) -- (1,2) -- (1.5,2.5) -- (3.5,2.5);
  \draw[-, very thick]
  (-.5,1.5) -- (.5,1.5) -- (1,2) -- (2,1) -- (2.5,1.5) -- (3.5,1.5);
  \draw[-, very thick]
  (-.5,0.5) -- (0,1.5) -- (.5,0.5) -- (1.5,0.5) -- (2,1) -- (3,0) -- (3.5,0.5);
  \draw[-, very thick]
  (-.5,-0.5) -- (0,-1.5) -- (.5,-0.5) -- (1.5,-0.5) -- (2,-1) -- (3,0) -- (3.5,-0.5);
  \draw[-, very thick]
  (-.5,-1.5) -- (.5,-1.5) -- (1,-2) -- (2,-1) -- (2.5,-1.5) -- (3.5,-1.5);
  \draw[-, very thick]
  (-.5,-2.5) -- (0,-1.5) -- (.5,-2.5) -- (1,-2) -- (1.5,-2.5) -- (3.5,-2.5);
\end{tikzpicture}
\quad = \quad
\begin{tikzpicture}[scale=.5,baseline=0]
  

  \foreach \i/\y in {3/2.5, 2/1.5, 1/0.5} {
    \node (src\i) at (-2.6,\y) {$\scriptstyle{\src\; \i}$};
    \node (snk\i) at (8.2,\y) {$\scriptstyle{\snk\; \i}$};
    \node (src-\i) at (-2.6,-\y) {$\scriptstyle{\src\; \ol{\i}}$};
    \node (snk-\i) at (8.2,-\y) {$\scriptstyle{\snk\; \ol{\i}}$};
  }
  \node (srcbul3) at (-1,2.45){$\bullet$};
  \node (srcbul2) at (-1,1.45){$\bullet$};
  \node (srcbul1) at (-1,0.45){$\bullet$};
  \node (srcbul-1) at (-1,-0.55){$\bullet$};
  \node (srcbul-2) at (-1,-1.55){$\bullet$};
  \node (srcbul-3) at (-1,-2.55){$\bullet$};
  \node (snkbul3) at (7,2.5){$\bullet$};
  \node (snkbul2) at (7,1.5){$\bullet$};
  \node (snkbul1) at (7,0.5){$\bullet$};
  \node (snkbul-1) at (7,-0.5){$\bullet$};
  \node (snkbul-2) at (7,-1.5){$\bullet$};
  \node (snkbul-3) at (7,-2.5){$\bullet$};
  \node (x1top) at (-0,1.5){$\bullet$};
  \node (x1bot) at (-0,-1.52){$\bullet$};
  \node (x2top) at (2,1.95){$\bullet$};
  \node (x2bot) at (2,-2.05){$\bullet$};
  \node (x3top) at (4,0.98){$\bullet$};
  \node (x3bot) at (4,-1.02){$\bullet$};
  \node (x4mid) at (6,-0.03){$\bullet$};

  \foreach \ymod in {1, -1} {
    \draw[-, thick, postaction={decorate, decoration={markings,mark=at position 0.12 with {\arrow[black, scale=1.32, ultra thick]{>}}, mark=at position 0.335 with {\arrow[black, scale=1.32, ultra thick]{>}}, mark=at position 0.75 with {\arrow[black, scale=1.32, ultra thick]{>}} }}]
    (-1,2.5*\ymod) -- (0,1.5*\ymod) .. controls (1,2.5*\ymod) .. (2,2*\ymod) .. controls (3,2.5*\ymod) .. (4, 2.5*\ymod) -- (7, 2.5*\ymod);

    \draw[-, thick, postaction={decorate, decoration={markings,mark=at position 0.1 with {\arrow[black, scale=1.32, ultra thick]{>}}, mark=at position 0.27 with {\arrow[black, scale=1.32, ultra thick]{>}}, mark=at position 0.52 with {\arrow[black, scale=1.32, ultra thick]{>}}, mark=at position 0.91 with {\arrow[black, scale=1.32, ultra thick]{>}} }}]
    (-1,1.5*\ymod) -- (0,1.5*\ymod) .. controls (1,1.5*\ymod) .. (2,2*\ymod) -- (4,1*\ymod) .. controls (5,1.5*\ymod) .. (6, 1.5*\ymod) -- (7,1.5*\ymod);

    \draw[-, thick, postaction={decorate, decoration={markings,mark=at position 0.12 with {\arrow[black, scale=1.32, ultra thick]{>}}, mark=at position 0.43 with {\arrow[black, scale=1.32, ultra thick]{>}}, mark=at position 0.78 with {\arrow[black, scale=1.32, ultra thick]{>}}, mark=at position 0.98 with {\arrow[black, scale=1.32, ultra thick]{>}} }}]
    (-1,0.5*\ymod) -- (0,1.5*\ymod) .. controls (1,0.5*\ymod) .. (2,0.5*\ymod) .. controls (3,0.5*\ymod) .. (4,1*\ymod) -- (6,0*\ymod) -- (7,0.5*\ymod);
  }
  
\end{tikzpicture}\ntnsp. 
\end{equation}

All networks (\ref{eq:BCconcat}) in $\net{BC}n$ can be drawn to exhibit mirror symmetry in a horizontal line separating the positively indexed sources and sinks from the negatively indexed sources and sinks. Alluding to this symmetry, we will refer to sources $j$ and $\ol j$, sinks $j$ and $\ol j$, and central vertices of the upper and lower stars of $F_{[a_i,b_i]}$ (for $a_i > 0$) in (\ref{eq:BCconcat}) as {\em reflections} of one another.  The central vertex of $F_{[\ol b, b]}$ is its own reflection. Similarly, we will refer to corresponding pairs of edges as reflections of one another.

Following \cite{BWHex}, \cite{CSkanTNNChar}, we
consider families
$\pi=(\pi_{\ol n}, \dotsc, \pi_{\ol 1}, \pi_1, \dotsc, \pi_n)$
of paths {\em covering} a star network $F \in \net{BC}n$, i.e., using all edges in $F$. Call $\pi$ a {\em $\msfBC$-path family} if
for each index $i \in [1,n]$, 
the reflections of the edges forming path $\pi_{i}$ are precisely the edges forming path $\pi_{\ol i}$. That is, $\pi_{\ol i}$ must be the reflection of $\pi_i$.
Define $\pi$ to have {\em type $u = u_{\ol n} \cdots u_{\ol 1} u_1 \cdots u_n \in \bn$} if for all $i$, path $\pi_i$ begins at source $i$ and terminates at sink $u_i$.  In this case, we write $\snk(\pi_i) = u_i$, or $\snk_F(\pi_i) = u_i$ if there is any danger of confusion. For $F \in \net{BC}n$ and $u \in \bn$, define the sets
\begin{equation}\label{eq:BCpathfams}
  \begin{gathered}
    \PiBC(F) = \{ \pi \,|\, \pi \text{ a $\msfBC$-path family covering } F \},\\
    \PiBC_u(F) = \{ \pi \in \PiBC(F) \,|\, \type(\pi) = u \}.
    \end{gathered}
\end{equation}

For example, three path families 

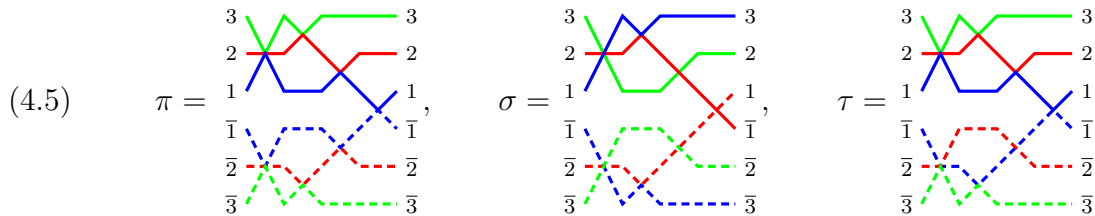
\begin{equation}\label{eq:pathfamex}
  \pi = 
\begin{tikzpicture}[scale=.5,baseline=0]
  \node at (-.9,2.5) {$\scriptstyle 3$};
  \node at (-.9,1.5) {$\scriptstyle 2$};
  \node at (-.9,0.5) {$\scriptstyle 1$};
  \node at (3.9,2.5) {$\scriptstyle 3$};
  \node at (3.9,1.5) {$\scriptstyle 2$};
  \node at (3.9,0.5) {$\scriptstyle 1$};
  \node at (-.9,-2.5) {$\scriptstyle {\ol 3}$};
  \node at (-.9,-1.5) {$\scriptstyle {\ol 2}$};
  \node at (-.9,-0.5) {$\scriptstyle {\ol 1}$};
  \node at (3.9,-2.5) {$\scriptstyle {\ol 3}$};
  \node at (3.9,-1.5) {$\scriptstyle {\ol 2}$};
  \node at (3.9,-0.5) {$\scriptstyle {\ol 1}$};
  \draw[-, very thick, green]
  (-.5,2.5) -- (0,1.5) -- (.5,2.5) -- (1,2) -- (1.5,2.5) -- (3.5,2.5);
  \draw[-, very thick, red]
  (-.5,1.5) -- (.5,1.5) -- (1,2) -- (2,1) -- (2.5,1.5) -- (3.5,1.5);
  \draw[-, very thick, blue]
  (-.5,0.5) -- (0,1.5) -- (.5,0.5) -- (1.5,0.5) -- (2,1) -- (3,0) -- (3.5,0.5);
  \draw[-, very thick, densely dashed, blue]
  (-.5,-0.5) -- (0,-1.5) -- (.5,-0.5) -- (1.5,-0.5) -- (2,-1) -- (3,0) -- (3.5,-0.5);
  \draw[-, very thick, densely dashed, red]
  (-.5,-1.5) -- (.5,-1.5) -- (1,-2) -- (2,-1) -- (2.5,-1.5) -- (3.5,-1.5);
  \draw[-, very thick, densely dashed, green]
  (-.5,-2.5) -- (0,-1.5) -- (.5,-2.5) -- (1,-2) -- (1.5,-2.5) -- (3.5,-2.5);
\end{tikzpicture}\ntnsp,
  \qquad 
  \sigma = \ntnsp
\begin{tikzpicture}[scale=.5,baseline=0]
  \node at (-.9,2.5) {$\scriptstyle 3$};
  \node at (-.9,1.5) {$\scriptstyle 2$};
  \node at (-.9,0.5) {$\scriptstyle 1$};
  \node at (3.9,2.5) {$\scriptstyle 3$};
  \node at (3.9,1.5) {$\scriptstyle 2$};
  \node at (3.9,0.5) {$\scriptstyle 1$};
  \node at (-.9,-2.5) {$\scriptstyle {\ol 3}$};
  \node at (-.9,-1.5) {$\scriptstyle {\ol 2}$};
  \node at (-.9,-0.5) {$\scriptstyle {\ol 1}$};
  \node at (3.9,-2.5) {$\scriptstyle {\ol 3}$};
  \node at (3.9,-1.5) {$\scriptstyle {\ol 2}$};
  \node at (3.9,-0.5) {$\scriptstyle {\ol 1}$};
  \draw[-, very thick, green]
  (-.5,2.5) -- (0,1.5) -- (.5,0.5) -- (1.5,0.5) -- (2,1) -- (2.5,1.5) -- (3.5,1.5);
  \draw[-, very thick, red]
  (-.5,1.5) -- (.5,1.5) -- (1,2) -- (2,1) -- (3,0) -- (3.5,-0.5);
  \draw[-, very thick, blue]
  (-.5,0.5) -- (0,1.5) -- (.5,2.5) -- (1,2) -- (1.5,2.5) -- (3.5,2.5);
  \draw[-, very thick, densely dashed, blue]
  (-.5,-0.5) -- (0,-1.5) -- (.5,-2.5) -- (1,-2) -- (1.5,-2.5) -- (3.5,-2.5);
  \draw[-, very thick, densely dashed, red]
  (-.5,-1.5) -- (.5,-1.5) -- (1,-2) -- (2,-1)  -- (3,0) -- (3.5,0.5);
  \draw[-, very thick, densely dashed, green]
  (-.5,-2.5) -- (0,-1.5) -- (.5,-0.5) -- (1.5,-0.5) -- (2,-1) -- (2.5,-1.5) -- (3.5,-1.5);
\end{tikzpicture}\ntnsp, \qquad
  \tau = \ntnsp
\begin{tikzpicture}[scale=.5,baseline=0]
  \node at (-0.9,2.5) {$\scriptstyle 3$};
  \node at (-0.9,1.5) {$\scriptstyle 2$};
  \node at (-0.9,0.5) {$\scriptstyle 1$};
   \node at (3.9,2.5) {$\scriptstyle 3$};
  \node at (3.9,1.5) {$\scriptstyle 2$};
  \node at (3.9,0.5) {$\scriptstyle 1$};
  \node at (-.9,-2.5) {$\scriptstyle \ol 3$};
  \node at (-0.9,-1.5) {$\scriptstyle {\ol 2}$};
  \node at (-0.9,-0.5) {$\scriptstyle {\ol 1}$};
  \node at (3.9,-2.5) {$\scriptstyle \ol 3$};
  \node at (3.9,-1.5) {$\scriptstyle {\ol 2}$};
  \node at (3.9,-0.5) {$\scriptstyle {\ol 1}$};
   \draw[-, very thick, green]
  (-.5,2.5) -- (0,1.5) -- (.5,2.5) -- (1,2) -- (1.5,2.5) -- (3.5,2.5);
  \draw[-, very thick, red]
  (-.5,1.5) -- (.5,1.5) -- (1,2) -- (2,1) -- (2.5,1.5) -- (3.5,1.5);
  \draw[-, very thick, blue]
  (-.5,0.5) -- (0,1.5) -- (.5,0.5) -- (1.5,0.5) -- (2,1) -- (3,0) -- (3.5,0.5);
  \draw[-, very thick, densely dashed, blue]
  (-.5,-0.5) -- (0,-1.5) -- (.5,-1.5) -- (1,-2) -- (2,-1) -- (3,0) -- (3.5,-0.5);
  \draw[-, very thick, densely dashed, red]
  (-.5,-1.5) -- (0,-1.5) -- (.5,-0.5) -- (1.5,-0.5) -- (2,-1) -- (2.5,-1.5) -- (3.5,-1.5);
  \draw[-, very thick, densely dashed, green]
  (-.5,-2.5) -- (0,-1.5) -- (.5,-2.5) -- (1,-2) -- (1.5,-2.5) -- (3.5,-2.5);
\end{tikzpicture}
\end{equation}

covering $F \defeq F_{[1,3]} \circ F_{[2,3]} \circ F_{[1,2]} \circ F_{[\ol1,1]}$
(\ref{eq:concatex})
satisfy 
$\pi \in \PiBC_{123}(F)$, $\sigma \in 
\PiBC_{3\ol12}(F)$, $\tau \not \in \PiBC(F)$, with $\tau$ failing to be a $\msfBC$-path family because $\tau_1, \tau_{\ol 1}$ in blue (also $\tau_2, \tau_{\ol 2}$ in red) are not reflections of one another.

By Lemma~\ref{l:longinversions}, we may interpret
$\ell(u)$ for all path families in $\PiBC_u(F)$ as follows.
\begin{obs}\label{o:pathinversions}
    Within 
    each path family $\pi = (\pi_{\ol n}, \dotsc, \pi_{\ol 1}, \pi_1, \dotsc, \pi_n) \in \PiBC_u(F)$, 
    there are 
    $\ell(u)$ pairs $(\pi_i, \pi_j)$ of paths satisfying $|i| < j$ and $\snk(\pi_j) < \snk(\pi_i)$. 
\end{obs}
\noindent
For example, in the path family $\sigma \in \PiBC_{3 \ol1 2}(F)$ in (\ref{eq:pathfamex}),
the $\ell(3\ol12) = 3$ pairs of paths satisfying
the conditions of the observation are $(\sigma_{\ol2}, \sigma_2)$, $(\sigma_1,\sigma_3)$, and $(\sigma_2, \sigma_3)$.

Each element $F \in \net{BC}n$ combinatorially interprets an element in $\hbnq$. To describe this interpretation, we extend the defect statistic~\cite{BWHex}, \cite{CSkanTNNChar}, \cite{Deodhar90} described in Section~\ref{sec: introduction}, using the inversions appearing in Lemma~\ref{l:longinversions}. Given a $\msfBC$-path family $\pi$ covering $F$, define a {\em type-$\msfBC$ defect} of $\pi$ to be
a triple $(\pi_i, \pi_j, k)$ with $|i| \leq j$, and $\pi_i$, $\pi_j$ meeting at one of the internal vertices of $F_{[c_k,d_k]}$ after having crossed an odd number of times. Let $\dfctbc(\pi)$ denote the number of type-$\msfBC$ defects of $\pi$.
We say that $F \in \net{\msfBC}n$ {\em graphically represents}
\begin{equation}\label{eq:Gtozbnq}
\sum_{\pi \in \PiBC(F)} \nTksp q^{\dfct^{\msfBC}(\pi)} T_{\type(\pi)}
\end{equation}
{\em as an element of $\hbnq$}. Refining the sets (\ref{eq:BCpathfams}) by counting defects of path families in them, let us define

\begin{equation}\label{eq:BCpathfams_type_defect}
\PiBC_{u,d}(F) = \{ \pi \in \PiBC_u(F) \,|\, \dfctbc(\pi) = d \}.
\end{equation}
Using this notation we may express (\ref{eq:Gtozbnq}) as
\begin{equation*}
    \sum_{u \in \bn} \sum_{d \geq 0} |\PiBC_{u,d}(F)| q^d T_u.
\end{equation*}

Observe that a single vertex can be the location of more than one defect. For example, consider the star network and $\msfBC$-path family

\begin{equation}\label{eq:all010}
  F_{[\ol2,2]} \circ F_{[\ol1,1]} \circ F_{[1,2]} \circ F_{[\ol2,2]} =
\begin{tikzpicture}[scale=.5,baseline=-5]
\node at (-.4,1.5) {$\scriptstyle{2}$};
\node at (-.4,0.5) {$\scriptstyle{1}$};  
\node at (-.4,-0.5) {$\scriptstyle{\ol1}$};
\node at (-.4,-1.5) {$\scriptstyle{\ol2}$};  
\node at (4.4,1.5) {$\scriptstyle{2}$};
\node at (4.4,0.5) {$\scriptstyle{1}$};  
\node at (4.4,-0.5) {$\scriptstyle{\ol1}$};
\node at (4.4,-1.5) {$\scriptstyle{\ol2}$};  
\draw[-, very thick] (0,1.5) -- (1,-1.5) -- (2,-1.5) -- (3,-.5) -- (3.5,0) -- (4,.5);
\draw[-, very thick] (0,-1.5) -- (1,1.5) -- (2,1.5) -- (3,.5) -- (3.5,0) -- (4,-.5);
\draw[-, very thick] (0,.5) -- (1,-.5) -- (2,.5) -- (3,1.5) -- (3.5,0) -- (4,1.5);
\draw[-, very thick] (0,-.5) -- (1,.5) -- (2,-.5) -- (3,-1.5) -- (3.5,0) -- (4,-1.5);
\end{tikzpicture},\qquad
\pi = 
\begin{tikzpicture}[scale=.5,baseline=-5]
\node at (-.4,1.5) {$\scriptstyle{2}$};
\node at (-.4,0.5) {$\scriptstyle{1}$};  
\node at (-.4,-0.5) {$\scriptstyle{\ol1}$};
\node at (-.4,-1.5) {$\scriptstyle{\ol2}$};  
\node at (4.4,1.5) {$\scriptstyle{2}$};
\node at (4.4,0.5) {$\scriptstyle{1}$};  
\node at (4.4,-0.5) {$\scriptstyle{\ol1}$};
\node at (4.4,-1.5) {$\scriptstyle{\ol2}$};  
\draw[-, very thick, red]
(0,1.5) -- (1,-1.5) -- (2,-1.5) -- (3,-.5) -- (3.5,0) -- (4,-.5);
\draw[-, very thick, blue]
(0,.5) -- (1,-.5) -- (2,.5) -- (3,1.5) -- (3.5,0) -- (4,1.5);
\draw[-, very thick, densely dashed, blue]
(0,-.5) -- (1,.5) -- (2,-.5) -- (3,-1.5) -- (3.5,0) -- (4,-1.5);
\draw[-, very thick, densely dashed, red]
(0,-1.5) -- (1,1.5) -- (2,1.5) -- (3,.5) -- (3.5,0) -- (4,.5);
\end{tikzpicture}.
\end{equation}
We have $\dfct^{\msfBC}(\pi) = 4$: defects of $\pi$ are $(\pi_{\ol1}, \pi_1, 2)$, $(\pi_{\ol1}, \pi_2, 3)$, $(\pi_1, \pi_2, 4)$, $(\pi_{\ol2}, \pi_2, 4)$.

\section{Main results}\label{s:main}

When one compares a concatenation $F'$ of $k-1$ simple star networks and a concatenation $F$ of $F'$ with one more simple star network, one sees bijections between certain sets of path families in $F$ and $F'$.

\begin{prop}\label{p:BCbijectiontoshorterpathreversals}
  Fix a sequence $(s_{[a_1,b_1]}, \dotsc, s_{[a_{k},b_k]})$
  of reversals in $\bn$, define the generator subset $J 
  = J_{[a_k,b_k]}$,
  and choose an element $w \in \Wejm$.
  Consider two type-$\msfBC$ star networks
  \begin{equation*}
    F' = F_{[a_1,b_1]} \circ \cdots \circ F_{[a_{k-1}, b_{k-1}]}
    \qquad\text{ and }\qquad
    F = F' \circ F_{[a_k,b_k]}.
  \end{equation*}
  Then for each element $v \in W_J$ and each $d \geq 0$, we have a bijection
\begin{equation}\label{eq:BCbijectionsreversals}
\Pi_{wv,d}(F)
    \quad \overset{1-1}{\longleftrightarrow} \ \
    \bigcup_{u \in W_J} 
    \Pi_{wu\ntnsp,\; d-\ell(u)}(F').
\end{equation}
  \end{prop}

  \begin{proof} 
  Fix $w$, $J$ as in the proposition.
  We demonstrate a bijection (\ref{eq:BCbijectionsreversals})
  in the case that $v = e$. Observe that each path family $\pi$ in $\Pi_w(F)$ can be decomposed uniquely as the concatenation of its truncation $\pi'$ to $F'$ with its truncation $\pi''$ to $F_{[a_k,b_k]}$: for some $u \in W_J$ we have 
  \begin{equation}\label{eq:pipi'pi''}
      \pi = \pi' \circ \pi'',\qquad
      \pi' \in \Pi_{wu}(F'), \qquad \pi'' \in \Pi_{u^{-1}}(F_{[a_k,b_k]}).
\end{equation}
      Let us write the first truncation map $\pi \mapsto \pi'$ with $w$ unspecified as $\trunc\ntnsp: \Pi(F) \rightarrow \Pi(F')$, and let us use $\psi$ to denote the restriction of $\trunc$ to $\Pi_w(F)$,
   
    \begin{equation}\label{eq:phidef}
    \psi:\Pi_w(F) \rightarrow \bigcup_{u \in W_J} \Pi_{wu}(F').
    \end{equation}

 We claim that $\psi$ is bijective. To see this, choose a path family $\tau$ belonging to $\cup_{u \in W_J} \Pi_{wu}(F')$ and let $y = y(\tau)$ be the element of $W_J$ such that $\type(\tau) = wy$. The set $\Pi_{y^{-1}}(F_{[a_k,b_k]})$ contains a unique element, call it $\sigma_{y^{-1}}$.
  Now define a map 
  \begin{equation}\label{eq:kappadef}
  \begin{aligned}
    \kappa: \bigcup_{u \in W_J} \Pi_{wu}(F') &\rightarrow \Pi_w(F)\\
    \tau &\mapsto \tau \circ \sigma_{y(\tau)^{-1}}.
    \end{aligned}
    \end{equation}
     It is easy to see that $\kappa$ inverts $\psi$. 

Now we claim that for $\pi \in \Pi_{w}(F)$, the path families $\pi$, $\pi' = \psi(\pi)$ satisfy

\begin{equation*}\dfct(\pi') = \dfct(\pi) - \ell(u).
\end{equation*}
To see this, recall that by Observation~\ref{o:pathinversions}, $\pi$ contains $\ell(w)$ pairs $(\pi_i,\pi_j)$ with $|i| < j$ and $\snk(\pi_j) < \snk(\pi_i)$.  By Lemma~\ref{l:minreps}, the condition $w \in \Wejm$ implies that the numbers $\snk(\pi_j)$ and $\snk(\pi_i)$ cannot both belong to the interval $[a_k,b_k]$
(or to the interval $[\ol{b_k},\ol{a_k}]$ if $a_k > 0$). Similarly, by (\ref{eq:Jfactor}), $\pi'$ contains $\ell(wu) = \ell(w) + \ell(u)$
pairs $(\pi'_i, \pi'_j)$ with $|i| < j$ and
\begin{equation}\label{eq:crossed}
    \snk_{F'}(\pi'_j) < \snk_{F'}(\pi'_i).
    \end{equation}
Thus the concatenation $\pi = \pi' \circ \pi''$ extends exactly
$\ell(u)$ such pairs to paths $(\pi_i,\pi_j)$ which cross in $F_{[a_k,b_k]}$ and by Lemma~\ref{l:minreps} satisfy

\begin{equation}\label{eq:uncrossed}
\snk_F(\pi_i) < \snk_F(\pi_j).
\end{equation}
Comparing (\ref{eq:crossed}) and (\ref{eq:uncrossed}), we see
that these crossings are defective and contribute $\ell(u)$ defects to $\pi$ beyond
the defects present in $\pi'$.

We complete the proof of the bijections (\ref{eq:BCbijectionsreversals}) by
demonstrating bijections between all pairs of sets $\Pi_{wv,d}(F)$ and  $\Pi_{wy,d}(F)$ with $v, y \in W_J$. In particular, for fixed $v, y \in W_J$, define the map $\phi_{v,y}: \Pi_{wv}(F) \rightarrow  \Pi_{wy}(F)$ by writing $\pi = \trunc(\pi) \circ \sigma_{z^{-1}}$ for some $z \in W_J$ and mapping
\begin{equation}\label{eq:phivy}
\pi \mapsto \mathrm{trunc}(\pi) \circ \sigma_{z^{-1}v^{-1}y}.
\end{equation}
To see that $\phi_{v,y}$ is well defined, observe that since the type of $\trunc(\pi)$ is $wvz$, the type of $\phi_{v,y}(\pi)$ is $wvz z^{-1}v^{-1}y = wy$.  To see that $\phi_{v,y}$ is invertible, observe that its inverse is $\phi_{y,v}$:
\begin{equation*}
\begin{aligned}
    \phi_{y,v}(\phi_{v,y}(\pi)) &= \phi_{y,v}(\trunc(\pi) \circ \sigma_{z^{-1}v^{-1}y})\\
    &= \trunc(\pi) \circ \sigma_{(z^{-1}v^{-1}y) y^{-1}v} \\
    &= \trunc(\pi) \circ \sigma_{z^{-1}}\\
    &= \pi.
    \end{aligned}
\end{equation*}

Finally, we claim that $\phi_{v,y}$ restricts to a bijection from $\Pi_{wv,d}(F)$ to $\Pi_{wy,d}(F)$ for all $d$. To see this, observe that as $z$ varies over $W_J$, the set of paths of $\trunc(\pi) \circ \sigma_{z^{-1}}$ which meet at the central vertex of $F_{[a_{k-1},b_{k-1}]}$ is constant, regardless of whether they cross there. Thus $\dfctbc(\pi' \circ \sigma_{z^{-1}})$ is also constant, and $\phi_v$ restricts to a bijection from $\Pi_{w,d}(F)$ to $\Pi_{wv,d}(F)$.
  \end{proof}

This leads to our main result: 
the star network $F_{[a_1,b_1]} \circ \cdots \circ F_{[a_k,b_k]}$ graphically represents the product $\wtc{s_{[a_1,b_1]}}q \cdots \wtc{s_{[a_k,b_k]}}q$ in the sense of (\ref{eq:Gtozbnq}).
 
\begin{thm}\label{t:BCwiringdiagramprod}
    Fix a sequence $(s_{[a_1,b_1]}, \dotsc, s_{[a_{k},b_k]})$
  of reversals in $\bn$ and define the type-$\msfBC$ star network

  \begin{equation}\label{eq:mainF}
    F = F_{[a_1,b_1]} \circ \cdots \circ F_{[a_{k}, b_{k}]}.
  \end{equation}
Then we have
\begin{equation}\label{eq:mainklprodk}
    \wtc{s_{[a_1,b_1]}}q \cdots \wtc{s_{[a_k,b_k]}}q = 
    \sum_{y \in \bn} \sum_{d \geq 0} |\Pi_{y,d}(F)| q^d T_y.
\end{equation}
\end{thm}
\begin{proof}
For any generator subset of the form $J = J_{[a,b]}$, we have by Proposition~\ref{p:wtcreversal} that the coefficient of $T_u$ in $\smash{\wtc{s_{[a,b]}}q}$ is $1$ for $u \in W_J$ and is $0$ otherwise.
On the other hand, the star network $F_{[a,b]}$ satisfies
\begin{equation*}
    |\Pi_{u,d}(F_{[a,b]})| = \begin{cases}
        1 &\text{if $u \in W_J$ and $d=0$},\\
        0 &\text{otherwise}.
    \end{cases}
\end{equation*}
Thus the identity (\ref{eq:mainklprodk}) holds for a single Kazhdan--Lusztig basis element indexed by a reversal.
Suppose therefore that the identity holds for products of $1, \dotsc, k-1$
such elements and consider a product of $k$ of them. Writing $F' = F_{[a_1,b_1]} \circ \cdots \circ F_{[a_{k-1}, b_{k-1}]}$, we have

\begin{equation}\label{eq:mainklprodkminus1}
        \wtc {s_{[a_1,b_1]}}q \cdots \wtc {s_{[a_{k-1},b_{k-1}]}}q = 
    \sum_{v \in \bn} \sum_{d \geq 0} |\Pi_{v,d}(F')| q^d T_v.
    \end{equation}

Defining $J = J_{[a_k,b_k]}$ and applying Proposition~\ref{p:klprodcosetexpansion} to (\ref{eq:mainklprodk}) -- (\ref{eq:mainklprodkminus1}) we have that the left-hand side of (\ref{eq:mainklprodk}) equals

\begin{equation}\label{equ:star}
    \sum_{w \in \Wejm} 
\ntksp \bigg( \sum_{u \in  W_{J}} \sum_{d\geq 0} |\Pi_{wu,d}(F')|  q^{\ell(u)+d}\bigg) T_w \wtc {s_{[a_k,b_k]}}q.
\end{equation}
Since $\Pi_{wu,c}(F')=\emptyset$ when $c<0$, the parenthesized expression in \eqref{equ:star} equals
\begin{equation*}
\begin{aligned}
    \sum_{u\in W_{\ntnsp J}}\sum_{\ d\geq -\ell(u)} \nTksp |\Pi_{wu,d}(F')|q^{\ell(u)+d}
    &=\sum_{u\in W_J}\sum_{d\geq 0} |\Pi_{wu,d-\ell(u)}(F')|q^{d}\\
    &= \sum_{d\geq 0}q^d\sum_{u\in W_J} |\Pi_{wu,d-\ell(u)}(F')|.
\end{aligned}
\end{equation*}
By Proposition~\ref{p:BCbijectiontoshorterpathreversals}, this simplifies to

\begin{equation*}
\sum_{d\geq 0}q^d|\Pi_{w,d}(F)|.
\end{equation*}
Plugging this into \eqref{equ:star} and writing $\wtc {s_{[a_k,b_k]}}q=\sum_{v\in W_J}T_v$, we have 
\begin{equation}\label{equ:double star}
\begin{aligned}
     \sum_{w \in \Wejm} 
\sum_{d\geq 0}q^d|\Pi_{w,d}(F)|T_w \ntksp \sum_{v\in W_J}\ntksp T_v
 = \sum_{w\in\Wejm} \sum_{v\in W_J}\sum_{d\geq 0}q^d|\Pi_{w,d}(F)|T_w T_v.
\end{aligned}
\end{equation}
By Proposition~\ref{p:BCbijectiontoshorterpathreversals}, we have $|\Pi_{w,d}(F)|=|\Pi_{wv,d}(F)|$ for each $v\in W_J$. By \eqref{eq:Jfactor} and \eqref{eq:tvtwtvw},
the final expression in \eqref{equ:double star} is equal to the right-hand side of \eqref{eq:mainklprodk}.
\end{proof}
As a consequence of Theorem~\ref{t:BCwiringdiagramprod}, we can prove a more general statement.  
First let us extend the set $\net\msfBC n$ to include more general planar networks.
Observe that sometimes in a
concatenation $G \circ F$, there exist vertices $x$ in $G$, $y$ in $F$ with $m(x,y) > 1$ 
edges incident upon both. Define the {\em condensed concatenation} $G \bullet F$ to be the subdigraph
of $G \circ F$ obtained by removing, for all such pairs $(x,y)$, all but one of the $m(x,y)$ edges
incident upon both, and by marking this edge with the multiplicity $m(x,y)$.
Let $\bnet{\msfBC}n$ denote
the union of simple star networks in $\net{\msfBC}n$ and all condensed concatenations of these. Let $\bcnet{\msfBC}n$ denote
the union of $\bnet{\msfBC}n$ and the set of ordinary concatenations of these. Call $\bcnet{\msfBC}n$ the set of {\em generalized star networks}, and observe that we have $\net{\msfBC}n \subset \bcnet{\msfBC}n$.

For $F \in \bcnet{\msfBC}n$, we define
the set $\PiBC(F)$ of $\msfBC$-path families coverin $F$ as before, with the additional requirement that
for each edge of multiplicity $k$ in $F$,
a path family $\pi \in \PiBC(F)$ must
contain $k$ paths $\pi_{i_1},\dotsc,\pi_{i_k}$ which include the given edge.
For example, one network in
the set $\bcnet{\msfBC}2$ and one path family
covering it are
\begin{equation}\label{eq:all010bullet}
  F_{[\ol2,2]} \circ F_{[\ol1,1]} \circ F_{[1,2]} \bullet F_{[\ol2,2]} =
\begin{tikzpicture}[scale=.5,baseline=-5]
\node at (-.4,1.5) {$\scriptstyle{2}$};
\node at (-.4,0.5) {$\scriptstyle{1}$};  
\node at (-.4,-0.5) {$\scriptstyle{\ol1}$};
\node at (-.4,-1.5) {$\scriptstyle{\ol2}$};  
\node at (4.4,1.5) {$\scriptstyle{2}$};
\node at (4.4,0.5) {$\scriptstyle{1}$};  
\node at (4.4,-0.5) {$\scriptstyle{\ol1}$};
\node at (4.4,-1.5) {$\scriptstyle{\ol2}$};  
\draw[-, very thick] (0,1.5) -- (1,-1.5) -- (2,-1.5) -- (3,-.5) -- (3.5,0) -- (4,.5);
\draw[-, very thick] (0,-1.5) -- (1,1.5) -- (2,1.5) -- (3,.5) -- (3.5,0) -- (4,-.5);
\draw[-, very thick] (0,.5) -- (1,-.5) -- (2,.5) -- (2.5,1) -- (3.5,0) -- (4,1.5);
\draw[-, very thick] (0,-.5) -- (1,.5) -- (2,-.5) -- (2.5, -1) -- (3.5,0) -- (4,-1.5);
\node at (3.1,.9) {$\scriptstyle{(2)}$};
\node at (3.1,-.9) {$\scriptstyle{(2)}$};
\end{tikzpicture},\qquad
\pi = 
\begin{tikzpicture}[scale=.5,baseline=-5]
\node at (-.4,1.5) {$\scriptstyle{2}$};
\node at (-.4,0.5) {$\scriptstyle{1}$};  
\node at (-.4,-0.5) {$\scriptstyle{\ol1}$};
\node at (-.4,-1.5) {$\scriptstyle{\ol2}$};  
\node at (4.4,1.5) {$\scriptstyle{2}$};
\node at (4.4,0.5) {$\scriptstyle{1}$};  
\node at (4.4,-0.5) {$\scriptstyle{\ol1}$};
\node at (4.4,-1.5) {$\scriptstyle{\ol2}$};  
\draw[-, very thick, red]
(0,1.5) -- (1,-1.5) -- (2,-1.5) -- (3,-.5) -- (3.5,0) -- (4,-.5);
\draw[-, very thick, blue]
(0,.5) -- (1,-.5) -- (2,.5) -- (2.4, .9)
-- (3.3,0) -- (3.5,0) -- (4,1.5);
\draw[-, very thick, densely dashed, blue]
(0,-.5) -- (1,.5) -- (2,-.5) -- (2.4, -.9)
-- (3.3,0) -- (3.5,0) -- (4,-1.5);
\draw[-, very thick, densely dashed, red]
(0,-1.5) -- (1,1.5) -- (2,1.5) -- (3,.5) -- (3.5,0) -- (4,.5);
\end{tikzpicture}.
\end{equation}
(Compare (\ref{eq:all010bullet}) to (\ref{eq:all010}).) 

In \cite[Thm.\,5.21]{SkanHyperGC} it was proved that for all $v \in \bn$ \avoidingp,
the Kazhdan--Lusztig basis element $\wtc vq$ of $\hbnq$ is graphically represented by a network $F_v$ in $\bnet{\msfBC}n$.
\begin{prop}\label{p:Fv}
    For each element $v \in \bn$ \avoidingp,
    there exists a network 
    \begin{equation}\label{eq:F_v}
    F_v = F_{[a_1,b_1]} \bullet \cdots \bullet F_{[a_p,b_p]}
    \end{equation}
    in $\bnet{\msfBC}n$ satisfying
    \begin{equation*}
        \wtc vq = \sum_{\pi \in \PiBC(F_v)} q^{\dfct^{\msfBC}(\pi)} T_{\type(\pi)} = \sum_{\pi \in \PiBC(F_v)} T_{\type(\pi)}.
    \end{equation*}
    In particular, for all $u \leq v$, the set $\PiBC(F_v)$ contains one path family of type $u$ and this path family is defect-free.  Furthermore,
    at most one interval $[a_i,b_i]$ in (\ref{eq:F_v}) satisfies $a_i = \ol{b_i},$ and if such an interval exists, we may assume that $i = 1$ or $i = p$.
    \end{prop}

The symmetry of elementary star networks of type $\msfBC$ guarantees that in the network $F_v$ (\ref{eq:F_v}), any edge labeled with a multiplicity has a reflection labeled with the same multiplicity. Assume that there are $2p = 2p(v)$ such edges,
and let the multiplicities of those edges which are incident upon at least one central vertex of an upper star of an 
elementary star network in (\ref{eq:F_v}) be \begin{equation}\label{eq:multiplicities}
m_{v,1},\dotsc,m_{v,p(v)}.
\end{equation}
These multiplicities and the $q$-analogs of their factorials relate the Kazhdan--Lusztig basis element $\wtc vq$ to the network 

\begin{equation}\label{eq:G_v}
G_v \defeq F_{[a_1,b_1]} \circ \cdots \circ F_{[a_p,b_p]}
\end{equation}
in $\net{\msfBC}n$ defined in terms of $F_v$ (\ref{eq:F_v}).
Recalling the standard definitions
\begin{equation*}
    [n]_q \defeq \begin{cases}
        1 + q + \cdots + q^{n-1} &\text{if $n \geq 1$},\\
        0 &\text{if $n=0$},
        \end{cases}
        \qquad
        [n]_q! \defeq
        \begin{cases}
[n]_q \cdots [1]_q
&\text{if $n \geq 1$},\\
1 &\text{if $n=0$},
        \end{cases}
\end{equation*}
and defining

\begin{equation}\label{eq:rvdef}
    r(v) \defeq [m_{v,1}]_q! \cdots [m_{v,p(v)}]_q!,
\end{equation}
in terms of (\ref{eq:multiplicities}), we have the following~\cite[Thm.\,4.3]{SkanNNDCB}.
\begin{lem}\label{l:rvcv}
For each element $v \in \bn$ \avoidingp,
the type-$\msfBC$ star network $G_v$ satisfies
\begin{equation}\label{eq:rvcv}
        r(v) \wtc vq = 
        \wtc{s_{[a_1,b_1]}}q \cdots \wtc{s_{[a_p,b_p]}}q = \nTksp \sum_{\pi \in \PiBC(G_v)} \nTksp q^{\dfct^{\msfBC}(\pi)} T_{\type(\pi)}.
\end{equation}
\end{lem}
\noindent In other words, the replacement of $G_v$ by $F_v$ on the right-hand side of (\ref{eq:rvcv}) leads to the equation
\begin{equation}\label{eq:EFidentity}
    \sum_{\pi \in \PiBC(G_v)} \nTksp q^{\dfct^{\msfBC}(\pi)} T_{\type(\pi)} = 
    r(v) \nTksp
    \sum_{\pi \in \PiBC(F_v)} \nTksp q^{\dfct^{\msfBC}(\pi)} T_{\type(\pi)}.
\end{equation}
More generally, for any generalized star networks $G$ and $F$, with $F$ formed by replacing multiple edges in $G$ by a single 
edge in $F$, 
path families in $F$ and $G$ are related in the following simple way.

\begin{lem}\label{l:collapse}
Fix a generalized star network $G$ in which $m$ multiplicity-$1$ edges are incident upon a pair $\{x,y\}$ of vertices and the $m$ reflections of these edges are incident upon the reflections $\{\ol x, \ol y\}$ of $x$ and $y$. Assume that at most one of the vertices $\{ x, y \}$ equals its reflection. Replace the edges incident upon $x$ and $y$
by a single multiplicity-$m$ edge,
replace their reflections by another such edge, and call the resulting generalized star network $F$.
Then we have
\begin{equation}\label{eq:mto1}
        \sum_{\pi \in \PiBC(G)} \nTksp q^{\dfct^{\msfBC}(\pi)} T_{\type(\pi)} = 
    [m]_q! \nTksp
    \sum_{\pi \in \PiBC(F)} \nTksp q^{\dfct^{\msfBC}(\pi)} T_{\type(\pi)}.
\end{equation}
\end{lem}
\begin{proof}
Assume that $x$ and $y$ lie on or above the horizontal line of symmetry in $G$. Let $\mathcal E = \{\varepsilon_1,\dotsc, \varepsilon_m\}$ be the set of $m$ edges from $x$ to $y$, labeled from bottom to top, and let $\ol{\mathcal E}$ be the set of reflections of these.  

Consider a path family $\pi \in \PiBC(G)$ with the property that
edges $\varepsilon_1,\dotsc,\varepsilon_m$ belong to
paths $\pi_{i_1},\dotsc,\pi_{i_m}$, 
respectively, with $i_1 < \cdots < i_m$.
Define 
$\Phi(\pi,\mathcal E)$ to be the set of all path families in $\PiBC(G)$ which agree with $\pi$ on every edge of $G$, except possibly for those edges belonging to $\mathcal E \cup \ol{\mathcal E}$.
Since each path family $\tau \in \Phi(\pi,\mathcal E)$ 
can differ from $\pi$
only in that 
edges $\epsilon_1,\dotsc,\epsilon_m$ belong to 
paths $\tau_{u_1},\dotsc,\tau_{u_m}$, respectively,
for some permutation $u_1 \cdots u_m$ of $\{i_1,\dotsc,i_m\}$, we have that
$\type(\tau) = \type(\pi)$ and 
$|\Phi(\pi,\mathcal E)| = m!$.

We claim that each path family $\tau \in \Phi(\pi,\mathcal E)$ has defects which agree with those of $\pi$ except possibly at the vertex $y$, where $\pi$ has no defect:
   \begin{equation}\label{eq:taupidefect}
    \dfct^{\msfBC}(\tau) = \dfct^{\msfBC}(\pi) + \# \{ \text{defects of $\tau$ occurring at vertex $y$}\}.
\end{equation}
To see this, observe that our choice of $\pi$ implies that each pair of paths among $\pi_{i_1},\dotsc,\pi_{i_m}$ crosses an even number of times before meeting at $y$.  Also the source-to-$x$ subpaths of $\pi_{i_1},\dotsc,\pi_{i_m}$ and $\tau_{i_1},\dotsc,\tau_{i_m}$ are identical. Now any vertex $z$ to the right of $y$
belongs to
paths $\tau_k$, $\tau_l$ if and only if it belongs to $\pi_k, \pi_l$.
Furthermore, if $(\pi_k, \pi_l)$ and $(\tau_k,\tau_l)$ appear in the same relative order between $x$ and $y$, then the number of crossings of $(\pi_k, \pi_l)$ preceding $z$ equals the number of crossings of $(\tau_k, \tau_l)$ preceding $z$. Otherwise these numbers of crossings differ by $2$, with one pair crossing at $x$ and at $y$ while the other pair crosses at neither of these vertices.
In either case, the meeting of $(\pi_k, \pi_l)$ at $z$ is defective if and only if the meeting of $(\tau_k, \tau_l)$ at $z$ is defective.

Consider therefore the defects occurring at vertex $y$. By our choice of $i_1,\dotsc,i_m$, 
the number of defects of $\tau$ occurring at vertex $y$ equals $\inv(u_1 \cdots u_m)$. Thus by (\ref{eq:taupidefect}) we have 

 \begin{equation}\label{eq:mqfactorial}
     \sum_{\tau \in \Phi(\pi,\mathcal E)} q^{\dfct^{\msfBC}(\tau)} = q^{\dfct^{\msfBC}(\pi)}\nTksp\nTksp
     \sum_{u \in \mfs{\{i_1,\dotsc,i_m\}}} \nTksp \nTksp q^{\inv(u)}
     = 
     [m]_q! q^{\dfct^{\msfBC}(\pi)},
 \end{equation}
 and the coefficient of $T_w$ on left-hand-side of (\ref{eq:mto1}) equals
 \begin{equation}\label{eq:nodefectaty}
     \sum_{\tau \in \PiBC_w(G)}\nTksp q^{\dfct^{\msfBC}(\tau)} =
     \nTksp\sumsb{\pi \in \PiBC_w(G)\\
     \text{having no }\\\text{defect at } y}
     \sum_{\tau \in \Phi(\pi,\mathcal E)} q^{\dfct^{\msfBC}(\tau)}
     = [m]_q! 
     \nTksp\sumsb{\pi \in \PiBC_w(G)\\ 
     \text{having no }\\\text{defect at } y}\nTksp 
     q^{\dfct^{\msfBC}(\pi)}.
 \end{equation}
 Now create network $F$ from $G$ by deleting edges $\varepsilon_2,\dotsc,\varepsilon_n$, $\ol{\varepsilon_2},\dotsc,\ol{\varepsilon_n}$, and by assigning multiplicity $m$ to $\varepsilon_1$ and $\ol{\varepsilon_1}$. 
 Similarly, create path family $\pi' \in \PiBC(F)$ from $\pi$ by replacing edges $\varepsilon_2,\dotsc,\varepsilon_m$ in $\pi_{i_2},\dotsc,\pi_{i_m}$
 with $\varepsilon_1$, and edges $\ol{\varepsilon_2},\dotsc,\ol{\varepsilon_m}$ 
 in $\pi_{\ol{i_1}},\dotsc,\pi_{\ol{i_m}}$
 with $\ol{\varepsilon_1}$. Specifically we have $\pi' \in \PiBC_{\type(\pi),\dfct(\pi)}(F)$. Moreover, it is easy to see that for each pair $(w,d) \in \bn \times \mathbb N$, this procedure defines a bijective correspondence between the sets
 $\{ \pi \in \PiBC_{w,d}(G) \,|\, \pi \text{ has no defect at } y \}$ and $\PiBC_{w,d}(F)$. Thus we may rewrite the final sum in (\ref{eq:nodefectaty}) as
\begin{equation*}
    [m]_q! \nTksp\sum_{\pi \in \PiBC_w(F)}\nTksp q^{\dfct^{\msfBC}(\pi)},
\end{equation*}
which is the coefficient of $T_w$ on the right-hand-side of (\ref{eq:mto1}).
\end{proof}

Now we may extend Theorem~\ref{t:BCwiringdiagramprod} to sequences of elements of $\bn$ whose type-$\msfBB$ and type-$\msfC$ Schubert varieties are simultaneously smooth.

\begin{thm}\label{t:BCwiringdiagramprod2}
Fix a sequence $(v^{(1)}, \dotsc, v^{(k)})$ of \pavoiding elements in $\bn$,
and define the generalized star network
\begin{equation}\label{eq:F}
  F = F_{v^{(1)}} \circ \cdots \circ F_{v^{(k)}}.
\end{equation}
Then we have
\begin{equation}\label{eq:cvk}
  \wtc{v^{(1)}}q \cdots \wtc{v^{(k)}}q = \sum_{y \in \bn}\sum_{d \geq 0} |\PiBC_{y,d}(F)|q^d T_y.
\end{equation}
\end{thm}
\begin{proof}
For $i = 1,\dotsc, k$, write 
\begin{equation*}
  E_{v^{(i)}} = F_{[a_{i,1}, b_{i,1}]} \circ \cdots \circ F_{[a_{i,p_i}, b_{i,p_i}]}
\end{equation*}
as in (\ref{eq:G_v})
and define
$r(v^{(i)})$             
as in (\ref{eq:multiplicities}),  (\ref{eq:rvdef}).  Then define the ordinary type-$\msfBC$ star network
\begin{equation}\label{eq:E}
    E = E_{v^{(1)}} \circ \cdots \circ E_{v^{(k)}}. 
\end{equation}
By Theorem~\ref{t:BCwiringdiagramprod} and Lemma~\ref{l:rvcv}, we have that
\begin{equation}\label{eq:rvkcvk}
    r(v^{(1)}) \cdots r(v^{(k)}) 
    \wtc{v^{(1)}}q \cdots \wtc{v^{(k)}}q = 
    \sum_{y \in \bn} \sum_{d \geq 0} |\PiBC_{y,d}(E)|q^d T_y,
\end{equation}
because both sides are equal to
\begin{equation*}
    (\wtc{s_{[a_{1,1},b_{1,1}]}}q \cdots \wtc{s_{[a_{1,p_1},b_{1,p_1}]}}q) 
    \quad \cdots \quad
    (\wtc{s_{[a_{k,1},b_{k,1}]}}q \cdots \wtc{s_{[a_{k,p_k},b_{k,p_k}]}}q).
\end{equation*}
By Lemma~\ref{l:collapse}, the right-hand side of (\ref{eq:rvkcvk}) equals
\begin{equation*}
r(v^{(1)}) \cdots r(v^{(k)})
\sum_{y \in \bn} \sum_{d \geq 0} |\PiBC_{y,d}(F)|q^d T_y,    
\end{equation*}
and we have the desired result.
    \end{proof}

It is possible to show that Proposition~\ref{p:Fv} and Theorem~\ref{t:BCwiringdiagramprod2} are not as strong as possible.  In particular, consider
the elements $3412$, $4231$ (in short one-line notation) of $\bn$, which clearly do not \avoidp.
Since 
reduced expressions for these elements do not contain the generator $s_0$, the 
Kazhdan--Lusztig basis elements $\wtc{3412}q$, $\wtc{4231}q$ factor just as their type-$\msfA$ analogs do, 
\begin{equation*}
    \wtc{3412}q = \wtc{s_2}q \wtc{s_1}q \wtc{s_3}q \wtc{s_2}q, \qquad
    \wtc{4231}q = \wtc{s_1}q \wtc{s_{[2,4]}}q \wtc{s_1}q.
\end{equation*}
(The factorization of $\wtc{3412}q$ follows from \cite[Thm.\,1]{BWHex}; the factorization of $\wtc{4231}q$ is a straightforward computation, observed in \cite[\S 4]{SkanNNDCB}.)
Now by Theorem~\ref{t:BCwiringdiagramprod} we may define networks
\begin{equation}\label{eq:patternstarnetworks}
F_{3412} \defeq F_{[2,3]} \circ F_{[1,2]} \circ F_{[3,4]} \circ F_{[2,3]},
\qquad
F_{4231} \defeq F_{[1,2]} \circ F_{[2,4]} \circ F_{[1,2]}
\end{equation}
which graphically represent
the Kazhdan--Lusztig basis elements
$\wtc{3412}q$, $\wtc{4231}q$ of $\hbnq$,
\begin{equation*}
    \wtc{3412}q = \nTksp \sum_{\pi \in \PiBC(F_{3412})} \nTksp
    q^{\dfct^{\msfBC}(\pi)}T_{\type(\pi)}, \qquad
    \wtc{4231}q = \nTksp \sum_{\pi \in \PiBC(F_{4231})} \nTksp
    q^{\dfct^{\msfBC}(\pi)}T_{\type(\pi)}. 
\end{equation*}
Moreover, this gives positive combinatorial formulas for the corresponding sets of  Kazhdan--Lusztig polynomials,
\begin{equation*}
    P_{u,3412}(q) = \nTksp \sum_{\pi \in \PiBC_u(F_{3412})}\nTksp q^{\dfct^{\msfBC}(\pi)}, 
    \qquad
    P_{u,4231}(q) = \nTksp \sum_{\pi \in \PiBC_u(F_{4231})}\nTksp q^{\dfct^{\msfBC}(\pi)}.
\end{equation*}
Thus we have the following consequence of Theorem~\ref{t:BCwiringdiagramprod}.

\begin{cor}\label{c:KLinterp}
    For $v \in \bn$ such that $\wtc vq$ is graphically represented by some network $F \in \bcnet{\msfBC} n$, the Kazhdan--Lusztig polynomials $\{ P_{u,v}(q) \,|\, u \leq v \}$ satisfy
    \begin{equation*}
        P_{u,v}(q) = \nTksp \sum_{\pi \in \PiBC_u(F)} \nTksp q^{\dfct^{\msfBC}(\pi)}.
    \end{equation*}
\end{cor}

\section{Open problems}\label{s:open}

Corollary~\ref{c:KLinterp} suggests several problems.
\begin{prob}\label{p:Cvfactor}
    Describe the elements $v \in \bn$ not \avoidingp{} 
    for which $\wtc vq$ is graphically represented by
    some network $F \in \bcnet{\msfBC}n$.
    \end{prob}
\noindent For elements $v \in \bn$ having this property, the Kazhdan--Lusztig basis element $\wtc vq$ necessarily factors as a product of other Kazhdan--Lusztig basis elements indexed by reversals, possibly divided by a polynomial in $q$.
\begin{prob}
    For $v$ as in Problem~\ref{p:Cvfactor},
    state an algorithm to construct $F_v$, equivalently, to factor $\wtc vq$.
\end{prob}
\noindent Preliminary results (e.g., \cite{DatSkan}) suggest that such factorizations do not exist for all $v \in \bn$.
\begin{prob}
    Prove that Kahzdan--Lusztig basis elements indexed by some subset of $\bn$ can not be graphically represented by generalized star networks. 
\end{prob}

It would also be interesting to extend the graphical interpretation of defects to Hecke algebras of Coxeter 
groups of other types.
\bp
State and prove a type-$\msfD$ analog of Theorem~\ref{t:BCwiringdiagramprod}.
\ep
\noindent Furthermore, perhaps
there is a nongraphical extension of Deodhar's subexpressions and defects which
applies to Hecke algebras of Coxeter groups of all types.

\bp

Find a suitable extension of the defect statistic that applies to an arbitrary Coxeter group $W$,
so that for all sequences $(J_1,\dotsc, J_k)$ of generator subsets of $W$, and maximal elements
$v^{(1)},\dotsc,v^{(k)}$ of 
$W_{J_1},\dotsc,W_{J_k}$, the coefficients in the expansion
\begin{equation*}
    \wtc{v^{(1)}}q \cdots \wtc{v^{(k)}}q = \sum_{w \in W} a_w T_w
    \end{equation*}
    may be interpreted as

    \begin{equation*}
        a_w = a_{w}(q) = \sum_{\sigma} q^{\dfct(\sigma)},
        \end{equation*}
        where the sum is over sequences
         $\sigma = (\sigma_1,\dotsc,\sigma_k) \in W_{J_1} \times \cdots \times W_{J_k}$ satisfying $\sigma_1 \cdots \sigma_k = w$.
      
\ep

\section{Acknowlegements}

The authors are grateful to Clement Cheneviere, Nate Kolo, Joshua Leonard, Roberto Mantaci, Sihong Pan, Celeste Stefanova, Angela Sun, Edward Sun, Diep Luong, Daniel Soskin, and Ben Spahiu for helpful conversations.

\bibliography{my}

\begin{thebibliography}{10}

\bibitem{BBDFaisceaux}
{\sc A.~Be\u{\i}linson, J.~Bernstein, and P.~Deligne}.
\newblock Faisceaux pervers.
\newblock In {\em Analysis and topology on singular spaces, {I} ({L}uminy, 1981)\/}, vol. 100 of {\em Ast\'{e}risque\/}. Soc. Math. France, Paris (1982), pp. 5--171.

\bibitem{BilleyLak}
{\sc S.~Billey and V.~Lakshmibai}.
\newblock {\em Singular loci of {S}chubert varieties\/}, vol. 182 of {\em Progress in Mathematics\/}.
\newblock Birkh\"auser Boston Inc., Boston, MA (2000).

\bibitem{BWHex}
{\sc S.~Billey and G.~Warrington}.
\newblock {Kazhdan--Lusztig} polynomials for $321$-hexagon-avoiding permutations.
\newblock {\em J. Algebraic Combin.\/}, {\bf 13}, 2 (2001) pp. 111--136.

\bibitem{BBCoxeter}
{\sc A.~Bj{\"o}rner and F.~Brenti}.
\newblock {\em Combinatorics of {C}oxeter groups\/}, vol. 231 of {\em Graduate Texts in Mathmatics\/}.
\newblock Springer, New York (2005).

\bibitem{CHSSkanEKL}
{\sc S.~Clearman, M.~Hyatt, B.~Shelton, and M.~Skandera}.
\newblock Evaluations of {H}ecke algebra traces at {K}azhdan--{L}usztig basis elements.
\newblock {\em Electron. J. Combin.\/}, {\bf 23}, 2 (2016).
\newblock Paper 2.7, 56 pages.

\bibitem{CSkanTNNChar}
{\sc A.~Clearwater and M.~Skandera}.
\newblock Total nonnegativity and {H}ecke algebra trace evaluations.
\newblock {\em Ann.\ Combin.\/}, {\bf 25} (2021) pp. 757--787.

\bibitem{DatSkan}
{\sc A.~Datko and M.~Skadera}.
\newblock Combinatorial interpretation of {Kazhdan--Lusztig} basis elements indexed by $45312$-avoiding permutations in $s_6$.
\newblock {\em Pure Math.\ Appl.\/}, {\bf 30} (2022) pp. 68--74.

\bibitem{Deodhar90}
{\sc V.~Deodhar}.
\newblock A combinatorial setting for questions in {Kazhdan--Lusztig} theory.
\newblock {\em Geom. Dedicata\/}, {\bf 36}, 1 (1990) pp. 95--119.

\bibitem{DougInv}
{\sc J.~Douglass}.
\newblock An inversion formula for relative {K}azhdan--{L}usztig polynomials.
\newblock {\em Comm. Algebra\/}, {\bf 18}, 2 (1990) pp. 371--387.

\bibitem{GJImm}
{\sc I.~Goulden and D.~Jackson}.
\newblock Immanants of combinatorial matrices.
\newblock {\em J. Algebra\/}, {\bf 148} (1992) pp. 305--324.

\bibitem{GreeneImm}
{\sc C.~Greene}.
\newblock Proof of a conjecture on immanants of the {Jacobi--Trudi} matrix.
\newblock {\em Linear Algebra Appl.\/}, {\bf 171} (1992) pp. 65--79.

\bibitem{KLSBasesQMBIndSgn}
{\sc R.~Kaliszewski, J.~Lambright, and M.~Skandera}.
\newblock Bases of the quantum matrix bialgebra and induced sign characters of the {H}ecke algebra.
\newblock {\em J. Algebraic Combin.\/}, {\bf 49}, 4 (2019) pp. 475--505.

\bibitem{KLRepCH}
{\sc D.~Kazhdan and G.~Lusztig}.
\newblock {Representations of Coxeter groups and Hecke algebras}.
\newblock {\em Invent. Math.\/}, {\bf 53} (1979) pp. 165--184.

\bibitem{KLSchub}
{\sc D.~Kazhdan and G.~Lusztig}.
\newblock {Schubert varieties and Poincar\'e duality}.
\newblock {\em Proc. Symp. Pure. Math., A.M.S.\/}, {\bf 36} (1980) pp. 185--203.

\bibitem{SkanNNDCB}
{\sc M.~Skandera}.
\newblock On the dual canonical and {Kazhdan}--{Lusztig} bases and 3412, 4231-avoiding permutations.
\newblock {\em J.\ Pure Appl.\ Algebra\/}, {\bf 212} (2008).

\bibitem{SkanHyperGC}
{\sc M.~Skandera}.
\newblock Hyperoctahedral group characters and a type-{BC} analog of graph coloring (2024).
\newblock Preprint {\tt math.CO/2402.04148} on {ArXiv}.

\bibitem{SpringerQACI}
{\sc T.~Springer}.
\newblock Quelques aplications de la cohomologie d'intersection.
\newblock In {\em S\'eminaire Bourbaki, Vol.\,1981/1982\/}, vol.~92 of {\em Ast{\'e}risque\/}. Soc. Math. France, Paris (1982), pp. 249--273.

\bibitem{StemImm}
{\sc J.~Stembridge}.
\newblock Immanants of totally positive matrices are nonnegative.
\newblock {\em Bull. London Math. Soc.\/}, {\bf 23} (1991) pp. 422--428.

\bibitem{StemConj}
{\sc J.~Stembridge}.
\newblock Some conjectures for immanants.
\newblock {\em Canad.\ J.\ Math.\/}, {\bf 44}, 5 (1992) pp. 1079--1099.

\end{thebibliography}

\vspace{2em}

\noindent$^\dag$ \textsc{Department of Mathematics, Lehigh University, Bethlehem, PA, USA}\\ \textit{Email address:} \href{mailto:gah219@lehigh.edu}{gah219@lehigh.edu}, \href{mailto:tjp225@lehigh.edu}{tjp225@lehigh.edu}, \href{mailto:mas906@lehigh.edu}{mas906@lehigh.edu}, \href{mailto:jiw922@lehigh.edu}{jiw922@lehigh.edu}

\end{document}